\documentclass[10pt,a4paper,leqno,english]{article}
\usepackage[T1]{fontenc}
\usepackage[latin1]{inputenc}
\usepackage{float}
\usepackage{graphics}
\usepackage{amssymb}

\usepackage{amsthm}
\usepackage{amsfonts}

\usepackage{float}
\usepackage{graphicx}
\usepackage{amsmath}

\newfont{\Blackboard}{msbm10 scaled 1200}

\newfont{\roma}{cmr10 scaled 1200}

\def \Z {{\mathbb{Z}}}
\def \R {{\mathbb{R}}}
\def \N {{\mathbb{N}}}
\def \C {{\mathbb{C}}}

\def \eps {{\epsilon}}

\def \lTU {{L^2(0,T;U)}}

\newcommand{\dint}{\displaystyle\int}
\newcommand {\nc}   {\newcommand}

\nc {\be}   {\begin{equation}} \nc {\ee}   {\end{equation}}
\nc {\beq}  {\begin{eqnarray}} \nc {\eeq}  {\end{eqnarray}}
\nc {\beqs} {\begin{eqnarray*}} \nc {\eeqs} {\end{eqnarray*}}

\newcommand{\bthe}{\begin{theorem}}
\newcommand{\ethe}{\end{theorem}}
\newcommand{\brk}{\begin{remark}}
\newcommand{\erk}{\end{remark}}
\newcommand{\bco}{\begin{corollary}}
\newcommand{\eco}{\end{corollary}}
\newcommand{\blem}{\begin{lemma}}
\newcommand{\elem}{\end{lemma}}
\newcommand{\bprop}{\begin{proposition}}
\newcommand{\eprop}{\end{proposition}}
\newcommand{\bdef}{\begin{definition}}
\newcommand{\bex}{\begin{example}}
\newcommand{\eex}{\end{example}}

\newcommand{\caD}{{\cal D}}

\newcommand{\caA}{{\cal A}}

\newcommand{\caL}{{\cal L}}

\newcommand{\caH}{{\cal H}}

\newcommand{\lb}{\lambda}

\newcommand{\da}{{\cal D}(A)}
\newcommand{\dac}{{\cal D}(\caA_c)}
\newcommand{\dad}{{\cal D}(\caA_d)}

\newcommand{\al}{\alpha}
\newcommand{\dsum}{\displaystyle \sum}

\def\edc{\end{document}}

\newtheorem{theorem}{Theorem}[section]
\newtheorem{lemma}[theorem]{Lemma}
\newtheorem{corollary}[theorem]{Corollary}
\newtheorem{remark}[theorem]{Remark}
\newtheorem{definition}[theorem]{Definition}
\newtheorem{proposition}[theorem]{Proposition}
\newtheorem{example}[theorem]{Example}
\def\dfrac{\displaystyle \frac }

\newcommand{\inta}{\dint_\Omega}

\newcommand{\intbn}{\dint_{\Gamma_1}}
\newcommand{\intta}{\int_0^T\dint_\Omega}
\newcommand{\inttb}{\int_0^T\dint_\Gamma}
\newcommand{\inttbn}{\int_0^T\dint_{\Gamma_1}}

\newcommand{\intabt}[1]{\left[\inta{#1}\right]_0^T}

\newcommand{\p}[1]{\partial_{#1}}
\newcommand{\g}{\nabla}

\newcommand{\w}{\widehat{w}_2}
\newcommand{\uu}{\widehat{u}_2}
\newcommand{\ww}{\widehat{w}_3}
\newcommand{\www}{\widehat{w}_4}
\newcommand{\pk}{\sqrt{2}\sin (k\pi x)}
\newcommand{\pkxi}{\sqrt{2}\sin (k\pi \xi)}
\newcommand{\pkxii}{|\sin (k\pi \xi)|^2}

\def \Z {{\mathbb{Z}}}
\def \R {{\mathbb{R}}}
\def \N {{\mathbb{N}}}
\def \C {{\mathbb{C}}}
\def \Q {{\mathbb{Q}}}

\begin{document}
\thispagestyle{empty}
\title{\bf Remark on stabilization of second order evolution equations by 
unbounded dynamic feedbacks and applications}
\author{Zainab Abbas $^{\dag}$, Ka\"{\i}s Ammari  
\thanks{UR Analyse et Contr\^ole des Edp, UR 13ES64, D\'epartement de Math\'ematiques, Facult\'e des Sciences de Monastir, Universit\'e de Monastir, 5019 Monastir, Tunisie,
e-mail: kais.ammari@fsm.rnu.tn,} \, and \,  Denis Mercier \thanks{Laboratoire de Math\'ematiques
et ses Applications de Valenciennes, FR CNRS 2956, Institut des Sciences et Techniques de
Valenciennes, Universit\'e de Valenciennes et du Hainaut-Cambr\'esis, Le Mont
Houy, 59313 VALENCIENNES Cedex 9, FRANCE, e-mail: (Zainab Abbas) zainab.abbas@univ-valenciennes.fr, (Denis Mercier) denis.mercier@univ-valenciennes.fr}}

\date{}
\maketitle

\begin{abstract}
In this paper we consider second order evolution equations with unbounded dynamic feedbacks. Under a regularity assumption we show that observability properties for the undamped problem imply decay estimates for the damped problem. We consider both uniform and non uniform decay properties.
\end{abstract}



\section{Introduction} \label{intro}
Let $X$ be a complex Hilbert space with norm and inner product denoted respectively by $\|.\|_X$ and $<.,.>_{X,X}$.
Let $A$ be a linear unbounded positive self-adjoint operator which is the Friedrichs extension of the
 triple $(X,V,a)$, where $a$ is a closed quadratic form with domain $V$ dense in $X.$ Note that by definition $\da$ (the domain of $A$) is dense in $X$ and  $\da$ equipped with  the graph norm is a Hilbert space and the embedding   $\da \subset  X $ is continuous.
Further, let $U$ be a complex Hilbert space (which will be identified with its dual space) with norm and inner
product respectively denoted by $\|.\|_U$ and  $<.,.>_{U,U}$ and let $B \in \caL (U,V^\prime),$ where $V^\prime$ is the dual space of $V$ obtained by means of the inner product in $X$.
Consider the system
\be 
\label{Gs0} 
\left \{ \begin{array}{lll}
x^{\prime \prime} (t)+ A x (t)+ B u(t)=0, &&t\in [0, +\infty )\\
\rho u^\prime (t)-\widehat{C} u(t)- B^* x^\prime (t)=0,&& t\in [0, +\infty ) \\
x(0)=x_0,x^\prime(0)=y_0, u(0)=u_0,  &&
\end{array}
\right.
\ee
with $\rho$ a scalar parameter. By replacing $\rho$ by $0$ and $-\widehat{C}$  by the identity
in system (\ref{Gs0}) we obtain
the system whose stability was studied in \cite{ammari:01}.

In this paper we are interested in studying the stability of linear control problems coming from elasticity which can be written as

\be \label{Gs}
\left \{ \begin{array}{lll}
x^{\prime \prime} (t)+ A x (t)+ B u(t)=0, &&t\in [0, +\infty )\\
u^\prime (t)-\widehat{C} u(t)- B^* x^\prime (t)=0,&& t\in [0, +\infty ) \\
x(0)=x_0,x^\prime (0)=y_0, u(0)=u_0,  &&
\end{array}
\right.
\ee
where $x: [0, +\infty ) \rightarrow X $ is the state of the system,
$u \in \lTU $ is the input function  and  $\widehat{C}$ is a $m-$dissipative operator on $U$. We denote the
differentiation with respect to time by $^\prime$.

The aim of this paper is to give sufficient conditions leading to the  uniform or non uniform stability of the solutions of the corresponding closed loop system.

The second equation of the considered system describes a dynamical control in some models. Some systems that can be covered by the
formulation  (\ref{Gs}) are for example the hybrid systems.

Let us finish this introduction with some notation used in the
remainder of the paper: the notation $A\lesssim B$ and $A\sim B$ means the
existence of positive constants $C_1$ and $C_2$, which are
independent of $A$ and $B$ such that $A\le
C_2 B$ and $C_1 B\le A\le C_2B$.

\section{Well-posedness results} \label{wpr}
In order to study the system (\ref{Gs}) we use a reduction order argument. First, we introduce the Hilbert space ${\caH}=V\times X \times U$ equipped with the scalar product
$$<z,\tilde{ z}>_{{\caH},{\caH}}=a(x,\tilde{x})+<y,\tilde{y}>_{X,X}+<u,\tilde{u}>_{U,U},\;\;\forall z,\tilde{z} \in {\caH},   z=(x,y,u),\tilde{z}=(\tilde{x},\tilde{y},\tilde{u}).$$
Then we consider the unbounded operator
$$\begin{array}{ll}
\caA_d:& \dad\longrightarrow \caH\\
    & z=(x,y,u)\longmapsto \caA_d z=(y,-Ax-Bu,B^*y+\widehat{C}u),
\end{array}
 $$
 where $$\dad=\lbrace (x,y,u) \in V\times V \times D(\widehat{C}), Ax+Bu \in X\rbrace.$$
So the system (\ref{Gs}) is formally equivalent to
\be \label{s1} z'(t)=\caA_d z(t), z(0)=z_0,
\ee
where $z_0=(x_0,y_0,u_0).$

\begin{proposition}\label{exist0}
The operator $\caA_d$ is an m-dissipative operator on $\caH$ and thus it generates a $C_0$-semigroup.
\end{proposition}
\begin{proof}
$$\begin{array}{lll}
 <\caA_d z,z>_{\caH,\caH}&=&a(y,x)-<Ax+Bu,y>_{X,X}+<B^*y+\widehat{C}u,u>_{U,U}\\
 &=&a(y,x)-a(x,y)-<Bu,y>_{V',V}+<B^*y,u>_{U,U}+<\widehat{C}u,u>_{U,U}\\
  &=&a(y,x)-a(x,y)+<\widehat{C}u,u>_{U,U}.\\
\end{array}
$$
Taking the real part of the above identity we get (\ref{en}) since $\widehat{C}$ is dissipative. Hence $\caA_d$ is dissipative.\\
We would like to show that there exists $\lambda>0$ such that  $\lambda I-\caA_d$ is surjective. Let $\lambda>0$ be given.
Clearly, we have   $\lambda\not\in\sigma(\widehat{C})$.
For $(f,g,h)\in \caH$, we look for $(x,y,u)\in \caD(\caA_d)$ such that $$(\lambda I-\caA_d)\begin{pmatrix}x\\
y\\
u\end{pmatrix}=\begin{pmatrix}f\\
g\\
h\end{pmatrix},$$
i.e. we are searching for $x\in V,\;y\in V,\; u\in D(\widehat{C})$ satisfying
$$\begin{array}{lll}
 \lambda x-y&=&f\\
 \lambda^2x+Ax+Bu&=&g+\lambda f\\
 (\lambda I-\widehat{C})u-B^*y&=&h.\\
\end{array}
$$
By Lax-Milgram lemma there exists a unique $x\in V$ such that
$$\left(\lambda^2+A+\lambda B(\lambda I-\widehat{C})^{-1}B^*\right)x=g+\lambda f+B(\lambda I-\widehat{C})^{-1}\left(B^*f-h\right).$$
In fact, we have  $\lambda^2+A+\lambda B(\lambda I-\widehat{C})^{-1}B^*\in \caL (V,V')$, $g+\lambda f+B(\lambda I-\widehat{C})^{-1}\left(B^*f-h\right)\in V'$
 and $$\Re\left<\left(\lambda^2+A+\lambda B(\lambda I-\widehat{C})^{-1}B^*\right)x,x\right>_{V',V}\ge \left<A x,x\right>_{V',V},$$
since
\begin{eqnarray*}\Re\left<B(\lambda I-\widehat{C} )^{-1}B^*x,x\right>_{V',V}&=&\Re\left<u,( \lambda I-\widehat{C})u\right>_{U,U}\\
&=&\lambda\|u\|^2-\Re\left<u,\widehat{C} u\right>_{U,U}\ge 0,\end{eqnarray*}
with $u=(\lambda I-\widehat{C} )^{-1}B^*x$, i.e. the coercivity property is satisfied. \\

Define $$u=(\lambda I-\widehat{C})^{-1}\left(h+B^*\left(\lambda x-f\right)\right),$$
by choosing $y=\lambda x-f$ we deduce the surjectivity of $\lambda I-A$. Finally, we conclude that  $\lambda I-A$ is bijective,
for all $\lambda >0$.\\
\end{proof}
Now, we are able to state the following existence result of problem (\ref{s1}).

\begin{proposition}\label{exist}
\noindent (i) For an initial datum $z_{0}\in \caH$, there exists a unique solution\\
$z\in C([0,\,+\infty),\, \caH)$
to system (\ref{s1}).
Moreover, if $z_{0}\in \dad$, then
\be \label{reg} z\in C([0,\,+\infty),\, \dad )\cap C^{1}([0,\,+\infty),\, \caH).\ee

(ii) For each $z_0\in \dad$, the energy $E(t)$ of the solution $z$ of (\ref{s1}), defined by $${E(t)=\dfrac{1}{2}\Vert z(t)\Vert_{\caH}^2,}$$
satisfies
\be \label{en} E^\prime (t)=\Re < \widehat{C}u(t),u(t)>\leq 0, \ee therefore the energy  is non-increasing.

Moreover, we have the following estimate
\be \label{est1} E(0)-E(t)=-\dint_0^t \Re< \widehat{C}u(s),u(s)>ds \leq \dfrac{1}{2} \|z_0\|_{\caH}^2,\,  \forall t\in [0,+\infty), \forall z_0\in \caH.\ee
\end{proposition}
\begin{proof}
$(i)$ is a direct consequence of  Lumer-Phillips theorem (see \cite{Pazy}).

$(ii)$ For an initial datum in $\dad$ from (\ref{reg}), we know  that $u$ is of class  $ C^{1}$ in time,
thus we can derive the energy
$E(t)$, and using Propostion \ref{exist0} we obtain:
$$E^\prime (t)=\Re<z^\prime,z>_{\caH,\caH}=\Re<\caA_d z,z>_{\caH,\caH}=\Re< \widehat{C}u,u>.$$
Hence the energy is non-increasing. Finally (\ref{est1}) is a direct consequence of (\ref{en}).
\end{proof}

Assume that $\widehat{C}$ can be written as $\widehat{C}=-C-DD^*$ where $C$ is a skew-adjoint operator on $U$,
$D\in {\cal L}(W,(D(C))^\prime)$, and  $W$ is supposed to be a Hilbert subspace of $U$ identified with its dual, thus $D^*\in {\cal L} (D(C),W) $.\\
Denote by  $\caA_c$  the operator obtained by replacing $\widehat{C}$ by $-C$ in the expression $\caA_d$. We can easily check that $\caA_c$ is closed
anti-symmetric, m-dissipative operator whose opposite $-\caA_c$ is also maximal dissipative, therefore $\caA_c$ is skew-adjoint and generates a unitary group.
Denote by  $\caA_r$ the   operator $$\caA_r:(x,y,u)\in \caH\mapsto (0,0,-DD^*u),$$
it is easy to see that $\caA_r$ is dissipative
and  $\caA_d=\caA_c+\caA_r.$ Note that the energy satisfies: 
\be
\label{Eprime} E^\prime (t)=-\|D^*u(t)\|_W^2.
\ee

\section{Some regularity results}
Let $T>0$ be fixed and $u\in L^2(0,T;U).$ Consider the evolution problem
\be \label{s3} z^\prime_2(t)=\caA_c z_2(t)+\caA_r z(t), z_2(0)=0,t\in[0,T],\ee
where $\caA_rz(t)=-(0,0,DD^*u(t)).$
\blem \label{estu2} Suppose that $D\in \caL(U)$. Then  problem (\ref{s3}) admits a unique solution $z_2(t)=(x_2(t),y_2(t),u_2(t))$ such that
$$u_2\in L^2(0,T;U),$$
satisfying the following  estimate
\be\label{est2}  \| D^* u_2\|_{L^2(0,T;U)}\le  c\| D^* u\|_{L^2(0,T;U)}, \ee
where $c$ is a positive constant.
\elem
\begin{proof} Clearly $\caA_rz(t)=-(0,0,DD^*u(t))\in C^1(0,T;\caH),$ and since $\caA_c$ generates a unitary group $e^{\caA_c.}$ on $\caH$, then (\ref{s3}) admits a unique solution given by
$$z_2(t)=\dint_0^t e^{\caA_c(t-s)} \caA_rz(s)ds=\dint_0^te^{\caA_c(s)}\caA_rz(t-s)ds,\forall t\in [0,T].$$
Moreover $D\in \caL(U)$ and 
$$\begin{array}{lll}
 \Vert  u_2\Vert^2_{L^2(0,T;U)}&=& \dint_0^T \Vert u_2(t)\Vert_U^2dt\\
&\leq& \dint_0^T \Vert z_2(t)\Vert_{\caH}^2 dt\\
&\leq& \dint_0^T \Vert\dint_0^t e^{\caA_c s}\; \caA_rz(t-s)ds\Vert_{\caH}^2 dt\\
&\leq& \dint_0^T \left(\dint_0^t \Vert e^{\caA_c s}\Vert \; \Vert \caA_rz(t-s)\Vert_{\caH} ds\right)^2 dt\\
&\leq& \dint_0^T (\dint_0^t \Vert \caA_rz(s)\Vert_{\caH} ds)^2 dt\\
&\leq& \dint_0^T (\dint_0^T \Vert \caA_rz(s)\Vert_{\caH} ds)^2 dt\\
&\leq& \dint_0^T (\dint_0^T \Vert DD^*u(s)\Vert_Uds)^2 dt\\
&\leq& \dint_0^T (\dint_0^T 1^2  ds)(\dint_0^T  \Vert DD^*u(s)\Vert_U^2   ds) dt\\
&\leq& \dint_0^T T \Vert D \Vert^2 \Vert D^*u\Vert_{L^2(0,T;U)}^2 dt \\
&\leq&  T^2 \Vert D \Vert^2 \Vert D^* u\Vert_{L^2(0,T;U)}^2.\\
\end{array}
$$
Consequently, as $\Vert D^*  u_2\Vert_{L^2(0,T;U)}\le\Vert D^* \Vert\Vert  u_2\Vert_{L^2(0,T;U)}$,
(\ref{est2}) holds with the constant $ T \Vert D \Vert\Vert D^* \Vert.$
\end{proof}

\section{Uniform stability}
In this section, we give sufficient and necessary condition  which lead to uniform stability of system (\ref{s1}).
We first introduce the conservative system associated with  the initial problem (\ref{Gs}) as
\be \label{Cs} 
\left \{ \begin{array}{lll}
x^{\prime \prime} (t)+ A x (t)+ B u(t)=0, &&t\in (0, +\infty )\\
 u^\prime (t) + C u(t)- B^* x^\prime(t)=0,&& t\in (0, +\infty ) \\
x(0)=x_0,x^\prime(0)=y_0, u(0)=u_0.  &&
\end{array}
\right.
\ee
The corresponding Cauchy problem can be written as
\be\label{s2}
 z^\prime (t)= \caA_c z(t), z(0)=z_0\in \dac.
\ee
Recall that the system (\ref{s2}) is the system (\ref{s1}) with $\widehat{C}=-C$ and that $(\caA_c,\dac)$ is given by 
$$\caA_c z=(y,-Ax-Bu,B^*y+Cu), \,\forall z=(x,y,u) \in \dac,$$ 
with
$$\dac=\lbrace (x,y,u) \in V\times V \times D(C), Ax+Bu \in X\rbrace.$$ Note also  that Proposition \ref{exist} still holds.
In order to get uniform stability we will need the following assumptions:

${\bf(O)}$ (Observability inequality) There exists a time $T>0$ and a constant $c(T)>0$ (which only depends on $T$) such that, for all $z_0\in \dac,$ the solution $z_1(t)=(x_1(t),y_1(t),u_1(t))$ of (\ref{s2}) satisfies the following observability estimate:
\be
\label{obs} 
\dint_0^T \|D^*u_1(s)\|_{W}^2ds \geq c(T) \|z_0\|_{\caH}^2.
\ee

${\bf({H})}$ (Transfer function estimate) Assume that for every $\lambda\in \C_+=\{\lambda\in\C|\Re \lambda>0\}$ 
$$\lambda\in \C_+\to H(\lambda)=-D^*(\lambda I+C+\lambda B^*(\lambda^2+A)^{-1}B)^{-1}D\in {\cal L}(W),$$
is bounded on $C_\beta=\{\lambda\in\C|\Re \lambda=\beta\}$, where $\beta$ is a positive constant.

\bthe\label{THexp}  Assume that assumption ${\bf(H)}$ is satisfied or that $D\in \caL(U)$. 
 Then  system (\ref{s1}) is exponentially stable, which means that the energy of the system satisfies
\be 
\label{exp1} 
E(t) \leq c \, e^{-\omega t} \, E(0), \forall t \in [0,+\infty),
\ee
where  $c$ and $\omega$ are two positive constants independent of the initial data $z_0\in \dad$ if and only if the inequality (\ref{obs}) is satisfied.
\ethe
By using \cite[Theorem 5.1]{miller:05} and \cite[Proposition 2.1]{amtuten} we have the following characterization of the uniform stabily of (\ref{s1}) by a frequency criteria (Hautus test).
\begin{corollary} \label{freqc}
Assume that assumption ${\bf(H)}$ is satisfied or that $D\in \caL(U)$. 
Then  system (\ref{s1}) is exponentially stable in the energy space if and only if  there exists a constant $\delta > 0$ such that for all $w \in \mathbb{R}, z \in \caD (\caA_c)$ we have
 \be
 \label{fc}
 \left\|(iw - \caA_c)z \right\|^2_{\caH} + \left\|\left( \begin{array}{cccl} 0 & 0 & D^* \end{array}\right) z \right\|^2_U \geq \delta \, \left\|z\right\|^2_\caH.
 \ee
\end{corollary}

\begin{proof} (of Theorem \ref{THexp}). Let $z(t)=(x(t),y(t),u(t))$ be the solution of (\ref{s1}) with initial datum $z_0\in \dad.$
Consider $z_1(t)=(x_1(t),y_1(t),u_1(t))$ the solution of (\ref{s2}) with initial datum $z_0\in \dad.$
Let  $z_2(t)=(x_2(t),y_2(t),u_2(t))$ be such that $z_2(t)=z(t)-z_1(t)$. Then  $z_2$ is solution of (\ref{s3}) and due to  Lemma \ref{estu2} its last component $u_2$  satisfies (\ref{est2}) if $D\in \caL(U)$. Otherwise, (\ref{est2}) holds true due to assumption  ${\bf(H)}$. 
Since $u=u_1+u_2,$ we get
$$\begin{array}{llll}
\|z_0\|_{\caH}^2 &\lesssim& \|D^*u_1\|_{L^2(0,T;W)}^2&\hspace{0.3cm} \mbox{ estimate } (\ref{obs})\\
&\lesssim & \|D^*u\|_{L^2(0,T;W)}^2+\|D^*u_2\|_{L^2(0,T;W)}^2&\hspace{0.3cm} \mbox{(triangle inequality)} \\
&\lesssim & \|D^*u\|_{L^2(0,T;W)}^2&\hspace{0.3cm} \mbox{(estimate (\ref{est2}))}. \\
\end{array}
$$
Indeed $x_2, u_2$ satisfies the system
\be \label{Ds1} 
\left \{ \begin{array}{lll}
x_2^{\prime \prime} (t)+ A x_2 (t)+ B u_2(t)=0, &&t\in (0, +\infty )\\
u_2^{\prime} (t)+C u_2(t)- B^* x_2^{\prime} (t)=-DD^*u(t),&& t\in (0, +\infty ) \\
x_2(0)=0,x_2^{\prime}(0)=0, u_2(0)=0.  &&
\end{array}
\right.
\ee
Extend $D^*u$ by zero on $\R\setminus[0,T]$. Since the system (\ref{Ds1}) is reversible by time we solve the system on $\R$. We obtain a
function $z\in C(\R;V)\cap C^1(\R;V)\cap L^2(\R;V)$ which is null for all $t\le 0$.
Let $\widehat{x}_2(\lambda)$, and $\widehat{u}_2(\lambda)$,  where $\lambda=\gamma+i\eta$, $\Re(\lambda)=\gamma>0$ and $\eta\in\R$, be the
respective Laplace transforms  of $x_2$ and $u_2$ with respect to $t$. Then $\widehat{x}_2$ and $\widehat{u}_2$ satisfy
\be \label{Ls} 
\left \{ \begin{array}{lll}
\lambda^2 \widehat{x}_2(\lambda)+A \widehat{x}_2(\lambda)+B\widehat{u}_2(\lambda)=0,&& \\
\lambda \widehat{u}_2(\lambda)+C \widehat{u}_2(\lambda)-B^*\lambda\widehat{x}_2(\lambda)=-DD^* \widehat{u}(\lambda).&&
\end{array}
\right.
\ee
Since $\lambda^2+A $ is invertible (Lax-Milgram lemma), we deduce from  the first equation of the system (\ref{Ls}) that
$$\widehat{x}_2=-(\lambda^2+A )^{-1}B\widehat{u}_2.$$
Substituting $\widehat{x}_2$ in the second equation of system (\ref{Ls}), we get
$$(\lambda I+C+\lambda B^*(\lambda^2+A)^{-1}B)\widehat{u}_2=-DD^* \widehat{u}.$$
Noting that the invertibility of $\lambda I+C+\lambda B^*(\lambda^2+A)^{-1}B$ follows from the invertibility of $\lambda I -\caA_c$ we obtain
$$D^*\widehat{u}_2=-[D^*(\lambda I+C+\lambda B^*(\lambda^2+A)^{-1}B)^{-1}D]D^* \widehat{u}$$ and by 
 assumption $\bf{(H)}$ estimate $(\ref{est2})$ holds. 
Finally, the inequality, $\|z_0\|_{\caH}^2 \lesssim\|D^*u\|_{L^2(0,T;U)}^2$, implies that there is a constant $c_1(T)$ which depends only of $T$ such that
$$ E(0)-E(T)\geq c_1(T) E(0).$$ But it is well known (see for instance \cite{ammari:01})  that the previous estimate is equivalent to (\ref{exp1}).  
\end{proof}

\section{Weaker decay}
In the case of non exponential decay in the energy space we give sufficient conditions for weaker decay properties. The statement of our second result requires some notations.

Let ${\caH}_1,{\caH}_2$ be two Banach spaces such that
$$\dad\subset {\caH}_1\subset {\caH}\subset {\caH}_2,$$
where
$$\|.\|_{\dad} \sim \|.\|_{{\caH}_1} $$
and
\be \label{interp}[{\caH}_1; {\caH}_2]_\theta ={\caH} \ee
for a fixed $\theta \in ]0; 1[,$  where $[: ; :]$ denotes the interpolation space (see for instance \cite{Triebel}).

Let $G : \R_+\longrightarrow \R_+$ be such that
$G$ is continuous, invertible, increasing on $\R_+$ and suppose that the function
$x\longmapsto \dfrac{1}{x^{\frac{\theta}{1-\theta}}}G(x)$ is increasing on $(0; 1)$.

\bthe \label{Thpolynomial} Assume that  the function $G$
satisfies the above assumptions  and that assumption ${\bf(H)}$ is satisfied or that $D\in \caL(U).$ 
Then the following assertions hold true:
\begin{enumerate}
\item
If for all non zero $z_0\in \dad,$ the solution $z_1(t)=(x_1(t),y_1(t),u_1(t))$ of (\ref{s2}) satisfies the following observability estimate:
\be\label{Obsv1}  \dint_0^T \|D^*u_1(s)\|_U^2ds \geq c(T) \|z_0\|_{\caH}^2 G\left(\dfrac{\|z_0\|_{{\caH}_2}^2}{\|z_0\|_{\caH}^2} \right),\ee  then
we have
\be\label{dp1} E(t)\lesssim \left[G^{-1}\left(\dfrac{1}{1+t}\right)\right]^{\frac{\theta}{1-\theta}}
\|z_0\|_{\dad}^2.\ee

\item
If for all non zero $z_0\in \dad,$ the solution $z_1(t)=(x_1(t),y_1(t),u_1(t))$ of (\ref{s2}) satisfies the following observability estimate:
\be \label{Obsv2}  \dint_0^T \|D^*u_1(s)\|_U^2ds \geq c(T) \|z_0\|_{{\caH}_2}^2,\ee  
then we have
\be\label{dp2}E(t)\lesssim \dfrac{1}{(1+t)^{\frac{\theta}{1-\theta}} }
\| z_0\| _{\dad}^2.\ee
\end{enumerate}
\ethe
\begin{proof} 
\begin{enumerate} 
\item 
Using the same arguments as in the proof of Theroem \ref{THexp} we get from (\ref{Obsv1})
$$ \forall z_0 \in \dad,\; \dint_0^T \|D^*u(s)\|_U^2ds \geq c(T) \|z_0\|_{\caH}^2 G\left(\dfrac{\|z_0\|_{{\caH}_2}^2}{\|z_0\|_{\caH}^2} \right).$$
The sequel follows the proof of Theorem 2.4 of \cite{ammari:01}, therefore we give the outlines below. 
Using (\ref{interp}) and the interpolation inequality 
$$\|z_0\|_{\caH} \leq  \|z_0\|_{{\caH}_1}^{1-\theta} \|z_0\|_{{\caH}_2}^{\theta}$$ we easily check   
$$\dfrac{\|z_0\|_{{\caH}_2}^{2}}{\|z_0\|_{\caH}^{2}}\geq 
\dfrac{
\|z_0\|_{\caH}^{\frac{2-2\theta}{\theta}}}
{\|z_0\|_{{\caH}_1}^{\frac{2-2\theta}{\theta}}},\; \forall z_0 \in \dad.$$ 
Consequently, using (\ref{Eprime}) and the fact that the function $t \mapsto \|z(t)\|_{\caH}$ is nonincreasing and $G$ is increasing we obtain the existence of a constant $K_1>0$ such that 
$$\|z(T)\|_{\caH}^2\leq \|z(0)\|_{\caH}^2-K_1 \|z(0)\|_{\caH}^2 G\left(\dfrac{
\|z(T)\|_{{\caH}}^{\frac{2-2\theta}{\theta}}}
{\|z(0)\|_{{\caH}_1}^{\frac{2-2\theta}{\theta}}}\right).$$ 
Applying the same arguments on successive intervals $[kT,(k+1)T],\; k=1,2,...$ we obtain the existence of a constant $K_2$ such that 
$$\|z((k+1)T)\|_{\caH}^2\leq \|z(k T)\|_{\caH}^2-K_2 \|z(kT)\|_{\caH}^2 G\left(\dfrac{
\|z((k+1) T)\|_{{\caH}}^{\frac{2-2\theta}{\theta}}}
{\|z(0)\|_{{\caH}_1}^{\frac{2-2\theta}{\theta}}}\right),\; \forall z_0 \in \dad.$$ 
If we set ${\cal E}_k=G\left(\dfrac{
\|z( k T)\|_{{\caH}}^{\frac{2-2\theta}{\theta}}}
{\|z(0)\|_{{\caH}_1}^{\frac{2-2\theta}{\theta}}}\right),$ the previous inequality, the property of $G$ and the fact that  $t \mapsto \|z( T)\|_{{\caH}}$ is nonincreasing then we get 
$$\dfrac{
\|z((k+1)T)\|_{\caH}^2}{\|z(k  T)\|_{\caH}^2}
\dfrac{{\cal E}_k}{{\cal E}_{k+1}}{\cal E}_k\leq
{\cal E}_k-K_2 {\cal E}_{k+1}^2.$$
Equivalently, we have 
\be\label{GG} \dfrac{
\frac{1}
{\left[\frac{\|z( k T)\|_{{\caH}}^{\frac{2-2\theta}{\theta}}}
{\|z( 0)\|_{{\caH}}^{\frac{2-2\theta}{\theta}}}\right]^{\frac{\theta}{1-\theta}}}
G\left(\dfrac{
\|z(k T)\|_{{\caH}}^{\frac{2-2\theta}{\theta}}}
{\|z(0)\|_{{\caH}_1}^{\frac{2-2\theta}{\theta}}}\right)  }
{\frac{1}
{\left[\frac{\|z(( k+1) T)\|_{{\caH}}^{\frac{2-2\theta}{\theta}}}
{\|z( 0)\|_{{\caH}}^{\frac{2-2\theta}{\theta}}}\right]^{\frac{\theta}{1-\theta}}}
G\left(\dfrac{
\|z(( k+1) T)\|_{{\caH}}^{\frac{2-2\theta}{\theta}}}
{\|z(0)\|_{{\caH}_1}^{\frac{2-2\theta}{\theta}}}\right)  }{\cal E}_{k+1}\leq {\cal E}_k-K_2 {\cal E}_{k+1}^2.\ee
Combining (\ref{GG}) and the fact that the function
$x\longmapsto \dfrac{1}{x^{\frac{\theta}{1-\theta}}}G(x)$ is increasing on $(0; 1),$
we get 
$$
{\cal E}_{k+1}\leq {\cal E}_k-K_2 {\cal E}_{k+1}^2.$$
We thus deduce  the existence of a constant $M>0$ such that ${\cal E}_k\leq \dfrac{M}{k+1}$ and we finally get (\ref{dp1}).
\item
As for $1.$ the proof is similar to  the second assertion of Theorem 2.4 of \cite{ammari:01} which is based on Lemma 5.2 of \cite{ammari:00} and is left to the reader. 
\end{enumerate}
\end{proof}

\section{Examples}
\subsection*{\underline{Beam System}}

We consider the following beam equation:
\be \label{Bs} 
\left \{ \begin{array}{lll}
u_{tt}(x,t)+u^{(4)}(x,t)=0, && 0<x<1, t\in [0, \infty )\\
\eta_t(t)+\beta\eta(t)-u_t(1,t)=0, && 0<x<1, t\in [0, \infty )\\
u(0,t)=u'(0,t)=u''(1,t)=0,&&  t\in [0, \infty ) \\
u'''(1,t)=\eta(t)&&\\ 
\end{array} 
\right. 
\ee
with the initial conditions
$$u(x,0)=u_0(x),u_t(x,0)=u_1(x),\eta(0)=\eta_0.$$

In this case
$$X=L^2(0,1), U=\C, V=\{u\in H^2(0,1):u(0)=u'(0)=0\},$$
$$\caD(A)=\{u\in H^4(0,1):u(0)=u'(0)=u''(1)=0,u^{(3)}(1)=0\},$$
$$a(u,v)=\int_0^1{\bar u}^{(2)}v^{(2)}dx \,(u,v \in V), \,Au=u^{(4)}\,(u\in \caD(A)),$$
$$B^*=\delta_1, B^*\varphi=\varphi(1)\, (\varphi \in V).$$
$$(B\eta,\varphi)_{V', V}=\bar\eta\varphi(1)\,(\eta\in\C, \varphi \in V),$$ and 
\beqs \widehat{C}:&\C\to \C&\\
&\eta\to-\beta\eta&,
\eeqs
where $\beta$ is a postive constant.\\

Note that $\widehat{C}$ is bounded, so we only need to find the observability inequality in order to deduce
the type of stability of the system. Since $B \in \caL (U,V')$ then $B\eta=\eta\cdot B1$, and since $B1\in V'$ 
and $ A\in \caL (V,V')$ then there exists a unique $u_0\in V$ such that $B1=Au_0$. Indeed, it is easy to check that 
$u_0(x)=\frac{x^2}{2}-\frac{x^3}{6}$. Moreover, $C=0$ and $D=\sqrt{\beta}$.

Remark that in this case $$\caD(\caA_c)=\caD(\caA_d)=\lbrace (u,v,\eta) \in V \times V \times \C  :  A u + B\eta \in L^2(0,1)\rbrace,$$ and 
\be\label{consbe}\caA_c\begin{pmatrix}
 u\\v\\\eta
 \end{pmatrix}=\begin{pmatrix}
 v\\-Au-B\eta\\B^*v
 \end{pmatrix}.
\ee  
Note that  since $\caD(A_c)$ is compactly injected in $\caH$, then $\caA_c$ has a compact resolvent and thus its spectrum is discrete. In addition, since $\caA_c$ is a skew-adjoint real operator, then its spectrum is constituted of pure imaginary conjugate eigenvalues. Now, let $\lambda=i\mu\in\sigma(\caA_c)$ with $U_\mu$ an associated eigenvector then $\bar\lambda=-i\mu\in\sigma(\caA_c)$ with $\bar{U}_\mu$ an associated eigenvector. Since the eigenvalues are conjugates , it is sufficient then to study  $\mu\ge 0$.
\begin{lemma}\label{lcheqb}
The eigenvalues of $\caA_c$ are algebraically simple. 
 Moreover, $0\in \sigma(\caA_c)$ and for every  $\lambda=i\mu\in\sigma(\caA_c), \mu > 0$,  $\mu$ satisfies the following characteristic equation,
\be\label{cheqb}f(\mu)=
{\mu}\sqrt{\mu}+{\mu}\sqrt{\mu}\cosh(\sqrt{\mu})\cos(\sqrt{\mu})+\sin(\sqrt{\mu})\cosh(\sqrt{\mu})-
\cos(\sqrt{\mu})\sinh(\sqrt{\mu})=0.\ee
\end{lemma}
\begin{proof}
First it is easy to see that  $0$ is a simple eigenvalue of $\caA_c$ and that an associated eigenvector is \\
$U=\eta(-u_0,0,1)^{\top}, \eta\in\C.$

Let $\lambda=i\mu\in\sigma(\caA_c), \mu > 0$, and let $U=(u,v,\eta)^\top\in \caD(\caA_c)$ be a nonzero associated eigenvector. Then $U$ satsifies $$\caA_c ( u, v, \eta)^{\top}=\lambda( u, v, \eta)^{\top},$$   which is equivalent to 
\be 
\left \{ \begin{array}{lll}
v=\lambda u\\
B^*v=\lambda\eta\\
-Au-\eta Au_0=\lambda v=\lambda^2 u.
\end{array} 
\right. 
\ee
We then deduce that 
$$A(u+\eta u_0)=-\lambda^2 u,  B^*u=\lambda u(1)=\eta.$$
But as $U\in \caD(\caA_c)$, then $Au+B\eta=A(u+\eta u_0)\in L^2(0,1)$, which implies that  $u+\eta u_0\in \caD(A)$ and that $u\in H^4(0,1)$
satisfies
 \begin{equation}\label{conds}u(0)=u'(0)=u''(1)=0, u'''(1)=\eta.\end{equation}
However, $A(u+\eta u_0)=(u+\eta u_0)^{(4)}= u^{(4)},$ thus we need to solve $u^{(4)}=-\lb^2 u=\mu^2u, \; u(1)=\eta$ with $u$ satisfying (\ref{conds}). 
We deduce that  $u$ could be written as $$u=c_1\sin(\sqrt{\mu}x)+c_2\sinh(\sqrt{\mu}x)+c_3\cos(\sqrt{\mu}x)+c_4\cosh(\sqrt{\mu}x),$$
with $C=(c_1,c_2,c_3,c_4)^\top$ satisfying 
\be \label{MCV} \widetilde{M}C=V_0\ee where
$$\widetilde{M}=\begin{pmatrix}0&0&1&1\\1&1&0&0\\-\sin(\sqrt{\mu})&\sinh(\sqrt{\mu})&-\cos(\sqrt{\mu})&\cosh(\sqrt{\mu})\\-\cos(\sqrt{\mu})&\cosh(\sqrt{\mu})&\sin(\sqrt{\mu})&\sinh(\sqrt{\mu})
\end{pmatrix}, V_0=\begin{pmatrix}
0\\0\\0\\\frac{\eta}{\mu\sqrt{\mu}}
\end{pmatrix}.
$$

We first remark that $\eta \neq  0.$ Otherwise, since  $u$ satisfies $u^{(4)}=\mu^2 u$ and the boundary conditions $u(1)=u''(1)=u'''(1)=0$, then there exists a constant $c\in \R$ such that $u$ is  given by $$u(x)=c(\sinh(\sqrt{\mu}(1-x))+\sin(\sqrt{\mu}(1-x)).$$ 
But $\cosh(\sqrt{\mu})+\cos(\sqrt{\mu})>0$,  then $u'(0)=0$  implies that $c=0$ and hence $U=(u,\lambda u, \eta)^\top=0$  which is a contradiction. \\

Consequently, each eigenvalue  of $\caA_c$ is simple. In fact, suppose to the contrary  that there exists $\mu\not=0$ such that $\lambda=i\mu$ is not algebraically simple. Then as 
$\caA_c$ is skew-adjoint, $\lambda=i\mu$ is not geometrically simple. Thus there exists at least two independent eigenvectors  $U_i=(u_i,v_i,\eta_i),i=1,2,$ corresponding to  $\lambda$, and hence $U=\eta_2 U_1 - \eta_1 U_2=(u,v,\eta)=(u,v,0)$ is an eigenvector which is impossible.

Going back to (\ref{MCV}), we get from the first three equations, 
$$c_2=-c_1,\; c_4=-c_3,\; c_3=-c_1\dfrac{\sin(\sqrt{\mu})+\sinh(\sqrt{\mu})}{\cos(\sqrt{\mu})+\cosh(\sqrt{\mu})}.$$
Therefore the last equation of (\ref{MCV}) becomes 
$$-\frac{2 c_1 (1+\cos(\sqrt{\mu })\cosh (\sqrt{\mu }))}{\cos(\sqrt{\mu })+\cosh(\sqrt{\mu })}=\dfrac{\eta}{\mu\sqrt{\mu}}.$$

As $\eta\not=0$ then the determinant of $\widetilde{M}$  which is  given by  $\det(\widetilde{M})=-2\left(1+\cos(\sqrt{\mu})\cosh(\sqrt{\mu})\right)$ is nonzero and $C$ is given by 
 $$C=\widetilde{M}^{-1}V_0=\frac{\eta}{2{\mu}\sqrt{\mu}(1+\cos(\sqrt{\mu})\cosh(\sqrt{\mu}))}\begin{pmatrix}
-\cos(\sqrt{\mu})-\cosh(\sqrt{\mu})\\\cos(\sqrt{\mu})+\cosh(\sqrt{\mu})\\\sin(\sqrt{\mu})+\sinh(\sqrt{\mu})\\-\sin(\sqrt{\mu})-\sinh(\sqrt{\mu})
\end{pmatrix}.$$
Substituting $C$  in the condition $u(1)=\eta$, we finally  deduce that  $\mu$ satisfies the charateristic equation (\ref{cheqb}).
\end{proof}

Now, we study the asymptotic behavior of the eigenvalues of $\caA_c$.
\begin{lemma} \label{asmbe}
There exists  $k_0\in \N$ large enough such that for all $k\ge k_0$ there exists one and only one $\lb_k=i\mu_k$ eigenvalue of $\caA_c$ with $\sqrt{\mu_k}\in [k\pi,(k+1)\pi]$.  Moreover, as $k\to\infty$, we have the following 
\be\label{tylb}\sqrt{\mu_k}=\frac{\pi}{2}+k\pi+\frac{1}{k^3\pi^3}+o\left(\frac{1}{k^3}\right).\ee
Let $U_{1,k}=(u_{1,k}, \lambda_{k} u_{1,k}, \eta_{1,k})$ be the associated normalized eigenvector. Then,
 \be\label{etab} |\eta_{1,k}|^2=\frac{4}{k^4}+o\left(\frac{1}{k^4}\right).\ee
\end{lemma}
\begin{proof}

\textbf{First step.} 
Let $z=\sqrt{\mu}$ where $i\mu\in \sigma(\caA_c)$ and $\mu>0$. Then  by (\ref{cheqb}), we have 
$$f(z^2)=z^3+\cosh z(z^3\cos z+\sin z)-\cos z \sinh z=0.$$
Replacing $\cosh z=\frac{e^z+e^{-z}}{2}$ and $\sinh z=\frac{e^z-e^{-z}}{2}$ in $f(z^2)$ and dividing by $\frac{z^3e^z}{2}$, we deduce that  $z$ satisfies $\tilde f(z)=0$, where
$$\tilde f(z)=\cos z +\frac{\sin z- \cos z}{z^3}+2e^{-z}+e^{-2z}\left(\cos z+\frac{\cos z}{z^3}+\frac{\sin z}{z^3}\right).$$

For $z$ large enough we have
$$\tilde f(z)=\cos z+ O(1/{z^3}).$$
It can be easily checked that  for $k$ large enough, $\tilde f$ doesn't admit any root outside the ball $B_k=B\left(z_k^0, \frac 1 {k^2}\right)$, with $z_k^{0}=\frac\pi 2+k\pi$.
Then by Rouch\'e's Theorem  applied on $B_k$,  we deduce that for $k$ large enough  there exists a unique root $z_k$ of $\tilde f$ in $[k\pi,(k+1)\pi]$. Moreover, $z_k$ satisfies  $$z_k=\frac{\pi}{2}+k\pi+\epsilon_k,$$
with $\epsilon_k=o(1)$.
 Since $z_k$ satisifes $\tilde f(z_k)=0$, then $\epsilon_k$ satisfies 
$$\cos \left(\frac{\pi}{2}+k\pi+\epsilon_k\right) +\frac{\sin \left(\frac{\pi}{2}+k\pi+\epsilon_k\right) +o(1)}{k^3\pi^3+o(k^3)}+O(e^{-z_k})=0.$$
Hence
$$-\sin (\epsilon_k)+\frac{\cos (\epsilon_k)}{k^3\pi^3}+o\left(\frac{1}{k^3}\right)=0,$$
and thus
$$-k^3\epsilon_k+o(k^3\epsilon_k^2)+\frac{1}{\pi^3}+o(k^2\epsilon_k)+o(1)=0,$$
which gives
$$\epsilon_k=\frac{1}{\pi^3k^3}+o(1/{k^3}).$$
Therefore, (\ref{tylb}) follows for $\mu_k=z_k^2$.

\textbf{Second step.} Set $\beta_k=\dfrac{\sin(z_k)+\sinh(z_k)}{\cos(z_k)+\cosh(z_k)}$.  Then
\be
\label{beta} \beta_k=\dfrac{e^{z_k}+2\sin(z_k)-e^{-z_k}}{e^{z_k}+2\cos(z_k)+e^{-z_k}}=1+o(e^{-z_k}).
\ee
By the proof of Lemma \ref{lcheqb}, the last component $\eta_k^1$ of $U_k^1$ is nonzero and thus 
$$U_k=(u_k, i z_k^2 u_k, 1)=\frac{1}{\eta_{1,k}}U_{1,k}$$
is an associated eigenvector to $iz_k^2$  with $u_k$ having the form,
$$u_k(x)=c_{1k}\sin({z_k}x)+c_{2k}\sinh({z_k}x)+c_{3k}\cos({z_k}x)+c_{4k}\cosh({z_k}x),$$ with
$$c_{2k}=-c_{1k},\; c_{4k}=-c_{3k},\; c_{3k}=-\beta_k c_{1k}.$$
It follows that 
\be\label{uu} u_k(x)=c_{1k}\left[(\sin({z_k}x)-\sinh({z_k}x)-\cos({z_k}x)+\cosh({z_k}x))+(\beta_k-1) (-\cos({z_k}x)+\cosh({z_k}x))\right].\ee

In order to get the behavior of $\eta_k = \frac{1}{\|U_k\|}$, it is enough to compute the integral $\int_{0}^{1} |u_k|^2 dx$. Indeed,  multiplying $u^{(4)}_k= - \lb^2 u_k = \mu^2 u_k$ by $\bar{u}_k$, integrating by parts and noting that
$u_k(0) = u^\prime_k (0) = 0
,\,\, u_k(1) = u^{\prime \prime \prime}_k (1)=1$ we obtain 
$$\int_{0}^{1}|u_k^{\prime \prime} |^2 dx={\mu_k}^2\int_{0}^{1} u_k^2 dx -1,$$
and hence
$$\|U_k\|^2=\int_{0}^{1} u_{kxx}^2 dx+{\mu_k}^2\int_{0}^{1} u_k^2 dx +1=2{\mu_k}^2\int_{0}^{1} u_k^2dx.$$
Since 
\begin{eqnarray*}
2 z_k^3c_{1k}&=&\frac{-\cos z_k - \frac{e^{z_k}}{2}(1+e^{-2z_k})}{1+ \frac{e^{z_k}}{2}\cos z_k (1+e^{-2z_k})}\\
&=&\frac{-1+O(e^{-k})}{(-1)^{k+1}\sin{\epsilon_k} +O(e^{-k})},
\end{eqnarray*}
we deduce that 
\be
\label{c1}c_{1k}=\frac{(-1)^{k}}{2}+o(1).
\ee

As
$$\int_0^1 (\sin(zx)-\sinh(zx)-\cos(zx)+\cosh(zx))^2dx=\int_0^1 (\sin(zx)-\cos(zx))^2dx+o(1)=1+o(1),$$
and
$$\int_0^1 (-\cos(zx)+\cosh(zx))^2dx=\frac{e^{2 z}}{8 z}+o(\frac{e^{2 z}}{8 z}),$$
we consequently deduce  due to  (\ref{beta}), (\ref{uu})  and (\ref{c1})  that
$$\int_0^1 u_k^2(x)dx =\dfrac{1}{4}+o(1), \text{ and }\,\|U_k\|^{2}=\dfrac{k^4}{4}+o(k^4).$$
Hence (\ref{etab}) holds.
\end{proof}
\bprop
\label{obsbe} Let $U_{1}=(u_1,v_1,\eta_1)^{T}$ be the solution of the conservative problem (\ref{consbe}) with initial datum $U_0\in \caD(\caA_c)$. Then there exists $ T>0$ and $c>0$ depending on $T$ such that\\
\be\label{obsbeam}\int_{0}^{T}|\eta_{1}(t)|^2dt\geq c\|U_0\|_{D(A^{-1})}^2.
\ee
\eprop

\begin{proof} We arrange the elements of $\sigma(\caA_c)$ in increasing order.\\
Let $J=\{i\mu: |\mu|<\mu_{k_0}\}$. Then $\sigma(\caA_c)=J\cup \{i\mu_k: |k|\ge k_0\}$
 and ${\displaystyle(U_{\mu})_{\mu\in J}} \cup {\displaystyle(U_{1,k})_{|k|\ge k_0}}$ forms a Hilbert basis of $\caH$. We may write 
 $$U_0=\displaystyle{\displaystyle{\sum_{\mu\in J}}u_0^{\mu}U_{\mu}+\sum_{|k|\ge k_0}}u_0^{(k)}U_{1,k}.$$ Moreover,
$$\eta_1(t)=\sum_{\mu\in J}u_0^{\mu}e^{i\mu t}\eta_{\mu}+\sum_{|k|\ge k_0}u_0^{(k)}e^{i\mu_kt}\eta_{1,k}.$$
Note that $\mu_{k+1}-\mu_k\geq\frac{\pi}{2}$ for $|k|\geq k_0$. Set $\gamma_0=\min\left\{\frac{\pi}{2},\min\{|\mu-\mu'|: \mu\in J,\mu'\in J\}\right\}$.
 As $|\mu-\mu'| \ge \gamma_{0}>0$ for all consecutive $\mu \in \sigma ( \caA_c ), {\mu'} \in \sigma (\caA_c)$. Then using Ingham's inequality there exists $T>{2\pi}{\gamma_0}>0$ and a constant $c$ depending on $T$ such that

$$\int_{0}^{T}|\eta_{1}(t)|^2dt\geq c\left(\sum_{\mu\in J}|u_{0}^{\mu}\eta_{\mu}|^2+\sum_{|k|\ge k_0 }|u_{0}^{(k)}\eta_{1,k}|^2\right).$$ 

Due to Lemma \ref{asmbe}, we have that $|\eta_{1,k}|^2\sim \dfrac{1}{k^4}$. we deduce using  Ingham's inequality the existence of  $T>0$ such that \begin{equation}\label{ob1be}\int_{0}^{T}|\eta_{1}|^2dt\gtrsim \sum_{\mu\in J} |u_0^{\mu}|^2|\mu|^{-2} +\sum_{|k|\ge k_0}\frac{|u_0^{(k)}|^2}{k^4}.\end{equation}
Therefore, we obtain (\ref{obsbeam}) as required.
\end{proof}

\begin{theorem}
Let $U_0\in \caD(\caA_d)$ and let $U$ be the solution  of the corresponding dissipative problem
$$U_t=\caA_d U, \,\, U(0)=U_0\in \caD (\caA_d).$$
Then $U$ satisfies,
\be\label{dpex}\|U(t)\|^2\lesssim \dfrac{1}{1+t}
\| U_0\|_{\caD(\caA_d)}^2.\ee
\end{theorem}
\begin{proof}
Since the operator $D\in\caL(U)$, then Lemma \ref{estu2} holds true.\\

Set  ${\caH}_1=\caD(\caA_c)$ and ${\caH}_2=\caD(\caA_c^{-1}),$ obtained by means of the inner product in $X$. Then $\caH=[{\caH}_1; {\caH}_2]_{1/2}$.
By Proposition \ref{obsbe}, we have $$ \dint_0^T \|D^*u_1(s)\|_U^2ds \geq c_T \|u_0\|_{{\caH}_2}^2.$$ By Theorem \ref{Thpolynomial} applied for $\theta=1/2$, we therefore obtain  (\ref{dpex}).
\end{proof}

\subsection*{\underline{Example on uniform stability}}
Consider the following system,
\be \label{thermo3} 
\left \{ \begin{array}{lll}
u_{tt}(x,t)+  u^{(4)}(x,t)+ \alpha \theta_{xx}(x,t)=0, &&t\in [0, \infty ), 0<x<1 \\
\theta_t (x,t)+\beta\theta(x,t)- \alpha u_{txx}(x,t)=0,&& t\in [0, \infty ), 0<x<1  \\
u(0,t)=u(1,t)=u''(0,t)=u''(1,t)=0&&  t\in [0, \infty )\\
u(x,0)=u_0(x),u_t(x,0)=u_1(x), \theta(x,0)=\theta_0(x), &&   0<x<1 
\end{array} 
\right. 
\ee
with $\alpha>0, \beta>0$.
Define the following spaces,
$$V=H^2(0,1)\cap H^1_0(0,1),X=U=L^2(0,1),$$
and the following operators,
$$\caD(A)=\{u\in H^4(0,1)\cap H^1_0(0,1):u_{xx}(0)=u_{xx}(1)=0\},\, Au=u_{xxxx}\in  L^{2}(0,1),$$
\beqs \widehat{C}:&L^2(0,1)\to L^2(0,1)&\\
&\theta\to\beta\theta.&
\eeqs
Remark that $\widehat{C}$ is a bounded operator on $L^2(0,1)$.
Moreover, $B$ and $B^*$ are given by
\beqs B:&U\to V'&,\,\, B^*:V\to U\\
&\,\,\,\,\,\theta\to\alpha\theta_{xx}&\,\,\,\, \,\,\hspace{0.7cm} u\to\alpha u_{xx},
\eeqs
and $D,D^*\in\caL(U)$ with $D\theta=D^*\theta=\sqrt{\beta}\theta$.
The norm defined on the energy space  $\caH=V\times X\times U$ is given by
$$\|(u,v,\theta)^{\top}\|^2_{\caH}=\int_{0}^{1}|u_{xx}|^2dx+\int_{0}^{1}|v|^2dx+\int_{0}^{1}|\theta|^2dx$$
We moreover have $$\caD(\caA_d)=\caD(\caA_c)=\{(u,v,\theta)^\top\in V\times V\times U:u^{(4)}+\theta_{xx}\in L^2(\Omega)\}.$$

The associated conservative system is given by\be \label{thermo3c} 
\left \{ \begin{array}{lll}
u_{tt}(x,t)+  u^{(4)}(x,t)+ \alpha \theta_{xx}(x,t)=0, &&t\in [0, \infty ), 0<x<1 \\
\theta_t (x,t)- \alpha u_{txx}(x,t)=0,&& t\in [0, \infty ), 0<x<1  \\
u(0,t)=u(1,t)=u''(0,t)=u''(1,t)=0&&  t\in [0, \infty )\\
u(x,0)=u_0(x),u_t(x,0)=u_1(x), \theta(x,0)=\theta_0(x), &&   0<x<1. 
\end{array} 
\right. 
\ee
In the following proposition we prove that the solution $u,\theta$ of (\ref{thermo3c}) satisfies the required observability inequality (assumption (O)), which is enough to deduce the exponential stability of (\ref{thermo3}) as $D\in\caL(U)$.
  \bprop \label{ObsUni} Let $U_0= (u_0,u_1,\theta_0)^{\top}\in\caH$. Then 
 the solution $(u,\theta)$ of (\ref{thermo3c}) satisfies
 \be\label{obsexp}\int_{0}^{T}|\theta(t)|^2dt\gtrsim \|U_0\|_{\caH}^2.
\ee  
  \eprop
  \begin{proof} Writing $(u_0,u_1,\theta_0)^{\top}\in\caD(\caA_c)$ with respect to the basis  ${\displaystyle(\sin(k\pi x))_{k\in\N^*}}$ of $L^2(0,1)$, we  have
$$u_0=\sum_{k\in\N^*}u_k^0\sin(k\pi x), u_1=\sum_{k\in\N^*}u_k^1\sin(k\pi x), \theta_0=\sum_{k\in\N^*}\theta_k^0\sin(k\pi x).$$
	The solution $(u,\theta)$ of (\ref{thermo3c}) is thus given by  $$u(t)=\displaystyle{\sum_{k\in\N^*}}u_k(t)\sin(k\pi x)\,\, \hbox{ and } \,\, \theta(t)=\displaystyle{\sum_{k\in\N^*}}\theta_k(t)\sin(k\pi x).$$ By the second equation of (\ref{thermo3c}), 
	$$\theta_k'(t)+\alpha k^2\pi^2u_k'(t)=0,\forall k\in\N^*.$$
	Due to  the initial conditions we deduce that 
	$$\theta_k(t)=-\alpha k^2\pi^2u_k(t)+\theta_k^0+\alpha k^2\pi^2u_k^0.$$
	Replacing $u$ and $\theta$ in the first equation of (\ref{thermo3c}), we deduce that 
	$$u_k''(t)+ k^4\pi^4(1+\alpha^2) u_k(t)=\alpha k^2\pi^2(\theta_k^0+\alpha k^2\pi^2u_k^0),\forall k\in\N^*,$$
hence 
$$u_k(t)=\frac{\alpha(\theta_k^0+\alpha k^2\pi^2u_k^0)}{k^2\pi^2(1+\alpha^2)}+c_1\cos(k^2\pi^2\sqrt{1+\alpha^2}t)+c_2\sin(k^2\pi^2\sqrt{1+\alpha^2}t),$$
where 
$$c_1=\frac{-\alpha\theta_k^0+k^2\pi^2u_k^0}{k^2\pi^2(1+\alpha^2)}, c_2=\frac{u_k^1}{k^2\pi^2\sqrt{1+\alpha^2}},$$
 obtained by the initial conditions $u_k(0)=u_k^0, u_k'(0)=u_k^1$ and $\theta_k(0)=\theta_k^0$.\\
 
Finally,
\begin{eqnarray*}
\theta_k(t)=\frac{1}{(1+\alpha^2)^{\frac{3}{2}}}[\sqrt{1+\alpha^2}(\theta_k^0+\alpha k^2\pi^2u_k^0)+
\alpha\sqrt{1+\alpha^2}(\alpha\theta_k^0- k^2\pi^2u_k^0)\cos(\sqrt{1+\alpha^2}k^2\pi^2 t)\\- \alpha(1+\alpha^2)u_k^1\sin(\sqrt{1+\alpha^2}k^2\pi^2 t)].
\end{eqnarray*}
Set $T=\dfrac{2}{\sqrt{1+\alpha^2}\pi}$. Then,
\begin{eqnarray*}
|\theta_k(t)|^2&=&\frac{1}{(1+\alpha^2)^{\frac{5}{2}}\pi}\left[(2+\alpha^4)(\theta_k^0)^2-2\alpha(-2+\alpha^2)k^2\pi^2\theta_k^0u_k^0+\alpha^2(3k^4\pi^4(u_k^0)^2+(1+\alpha^2)(u_k^1)^2)\right]\\
&=&\frac{\alpha^2(1+\alpha^2)(u_k^1)^2}{(1+\alpha^2)^{\frac{5}{2}}\pi}+\begin{pmatrix}
k^2u_k^0 & \theta_k^0
\end{pmatrix}M \begin{pmatrix}
k^2u_k^0 \\ \theta_k^0
\end{pmatrix},
\end{eqnarray*}
where $M$ is a square matrix given by 
	$$\begin{pmatrix}
\frac{3\alpha^2\pi^3}{(1+\alpha^2)^{\frac{5}{2}}}&-\frac{\alpha(-2+\alpha^2)\pi}{(1+\alpha^2)^{\frac{5}{2}}} \\
-\frac{\alpha(-2+\alpha^2)\pi}{(1+\alpha^2)^{\frac{5}{2}}} & \frac{2+\alpha^4}{\pi(1+\alpha^2)^{\frac{5}{2}}}\\
\end{pmatrix}.$$
But as $$\det M= \frac{2\alpha^2\pi^2}{(1+\alpha^2)^3}>0,
\text{trace}\,M=\frac{2+\alpha^4+3\alpha^2\pi^4}{\pi(1+\alpha^2)^{\frac{5}{2}}}>0,$$ we deduce that $\lb_{\text{min}}\geq c > 0$ (where $\lambda_{min}$ is the smallest eigenvalue of $M$) for some constant $c$ independent of $k$ and hence 
$$\int_{0}^{T}|\theta_k(t)|^2dt\ge T\left(\frac{\alpha^2(1+\alpha^2)(u_k^1)^2}{(1+\alpha^2)^{\frac{5}{2}}\pi}+ \lb_{\text{min}}(M)(k^4(u_k^0)^2+(\theta_k^0)^2)\right)$$
we get \begin{equation*}
\int_{0}^{T}|\theta(t)|^2dt\gtrsim \sum_{k\in\N^*}(k^4(u_k^0)^2+(\theta_k^0)^2+(u_k^1)^2)\gtrsim \|U_0\|_{\caH}^2.
\end{equation*}
We hence conclude (\ref{obsexp}) by denseness of $\caD(\caA_c)$ in $\caH$.
\end{proof}
 Recall that the energy of $(u,\theta)$  a solution of (\ref{thermo3c}) is defined by
$$E(t)=\frac 12\left(\int_{0}^{1}|u_{xx}|^2dx+\int_{0}^{1}|u_t|^2dx+\int_{0}^{1}|\theta|^2dx\right).$$
 \bthe
 Let $U_0\in\caH$. Then there exists $\omega>0$ such that the energy of the solution $(u,\theta)$ of (\ref{thermo3c}) satisfies \be \label{expexp}
E(t) \lesssim  e^{-\omega t} \, E(0), \forall t \in [0,+\infty).
\ee
 \ethe
 \begin{proof} 
By Proposition \ref{ObsUni}, assumption (O) holds true. Then (\ref{expexp}) follows by applying Theorem \ref{THexp}. 
 \end{proof}
\subsection*{\underline{Hybrid example-2D problem}}

Let $\Omega$ be a bounded domain of $\R^2$ whose boundary $\Gamma$ satisfies
$$\Gamma={\Gamma}_0\cup {\Gamma}_1
,\, {\bar{\Gamma}}_0\cap {{\bar{\Gamma}}}_1=\phi,
\, \hbox{ and }  \hbox{ meas }\Gamma_0\neq 0.$$
We assume moreover that there exists a point $x_0\in\R^2$ such that
$$\Gamma_0=\{x\in\Gamma: m(x).\nu\le 0\},\,\,\Gamma_1=\{x\in\Gamma: m(x).\nu\ge \omega >0\},$$
for some constant $\omega>0$,
where $m(x)=x-x_0$  and 
 $\nu=\nu(x)$ denotes the unit outward normal vector at $x\in \Gamma$. Denote by $R=\|m\|_\infty=
\displaystyle{\sup_{x\in \Omega}}\|m(x)\|$.\\
Consider the following system,
$$
\leqno(P_b)\hspace{4cm}
\left \{
\begin{array}{ll}
y_{tt}(x,t)-\Delta y(x,t)=0,&\, x\in\Omega,t>0,\\
y(x,t)=0,&\, x\in\Gamma_0, t>0,\\
a y_{tt}(x,t)+\p{\nu} y(x,t)+ \eta(x,t)=0,&\,x\in\Gamma_1, t>0,\\
\eta _t(x,t)-y_t(x,t)+b\eta (x,t)=0,&\, x\in\Gamma_1, t>0,\\
y(x,0)=y_0(x),\, y_t(x,0)=y_1(x), &  \,x\in\Omega,\\
\eta(x,0)=\eta_0(x)&\,x\in\Gamma_1,\\
\end{array}
\right.\\
$$
where $a$ and $b$ are two positive constants.
In order to justify that the system could be written in the proposed general form, we
introduce a proper functional setting. 
Let $$X=L^2(\Omega) \times L^2(\Gamma_1)$$
endowed with the inner product,
$$\left((y,\xi)^{\top},(\tilde{y},\tilde{\xi})^{\top}\right)_X =\inta <y,\tilde{y}>dx+\dfrac{1}{a} \intbn  <\xi, \tilde{\xi}> ds$$
and $$W=\{y\in H^{1}(\Omega):y=0 \text{ on } \Gamma_0\}=H_{\Gamma_0}^1(\Omega),\, \, U=L^2(\Gamma_1).$$
Define also $V$ by $$V=\{(y,\xi)\in W\times L^2(\Gamma_1): ay=\xi \text{ on } \Gamma_1\},$$ and the operator $(A,\caD(A))$ by 
$$A(y,\xi)^{\top}=(-\Delta y,\partial_{\nu}y|_{\Gamma_1})^{\top}$$
with
$$\caD(A)=\{x=(y,\xi)^\top\in V: y\in H^2(\Omega)\}.$$
 We can easily check using Lax-Milgram lemma that $(A\pm iI)$ are surjective. In addition, since $A$ is symmetric  we deduce that $A$ is  self-adjoint. The corresponding form $\tilde{a}$  is  given by 
$$\tilde{a}(u,\tilde{u})=\inta <y_x,\tilde{y}_x>dx,\, u=(y,\xi)^\top\in V,\tilde{u}=(\tilde{y},\tilde{\xi})^\top \in V.$$
In addition, we define for every $\eta\in U$ and $(y,\xi)^{\top}\in V$ the operators $B$ and $B^*$ by 
$$B\eta=(0,\eta)^\top, \, B^*(y,\xi)^{\top}=y|_{\Gamma_1}.$$
The operator $C=0$ and the operator $\widehat{C}$ is given by
$$\widehat{C}\eta=-b\eta, \, \eta\in L^2(\Gamma_1).$$
Hence the system $(P_b)$ can be written in the form of system (\ref{Gs}).

Accordingly, we define the energy space 
$$\caH=V \times L^2(\Omega) \times L^2(\Gamma_1)^2,$$ endowed with the inner product 
$$(u,\tilde{u})_\caH =\inta <y_x,\tilde{y_x}> dx + \inta <z,\tilde{z}> dx +\dfrac{1}{a} \intbn  <\xi, \tilde{\xi}> ds+ \intbn <\eta ,\tilde{\eta}>ds,$$
where $u=(y,\zeta,z,\xi,\eta),\,\tilde{u}=(\tilde{y},\tilde{\zeta},\tilde{z},\tilde{\xi},\tilde{\eta})\in \caH,$ and $<.,.>$ represents the Hermitian product in $\C$. The associated norm will be denoted by $\Vert \cdot\Vert_{\caH}$. Moreover,  $(\caA_d,\caD(\caA_d))$ is then given by
$$\caA_d u =(z,\xi, \Delta y, -\p \nu y-\eta,z|_{\Gamma_1}-b\eta),\, \forall u=(y,\zeta,z,\xi,\eta) \in {\mathcal D}({\mathcal A_d}),
$$ with
$$
\caD(\caA_d)=\{u=(y,\zeta,z,\xi,\eta) \in \caH : y\in H^2(0,1), z\in W, \zeta=a y|_{\Gamma_1} \xi= a z|_{\Gamma_1}\}.$$

Hence, the previous problem $(P_b)$ is formally equivalent to
\be \label{Pb}
\hspace{4cm}
 u_t=\caA_d u, \,\,\, u(0)=u_0,
\ee
where $u_0=(y_0,a y_0|_{\Gamma_1}, y_1, a y_1|_{\Gamma_1},\eta_0)$. 
The energy of the system $(P_b)$ is given by 
$$E(t)=\frac 12\left(\inta |y_t|^2 dx+ \inta |\g y|^2 dx+ \frac1a \intbn |y_t|^2 ds+\intbn |\eta ^2| ds\right),$$ 
and its derivative
$$\frac{d}{dt}E(t)=-b\intbn |\eta ^2| ds.$$

The corresponding conservative system is defined by
$$
\leqno(P_0)\hspace{4cm}
\left \{
\begin{array}{ll}
y_{tt}(x,t)-\Delta y(x,t)=0,&\, x\in\Omega,t>0,\\
y(x,t)=0,&\, x\in\Gamma_0, t>0,\\
a y_{tt}(x,t)+\p{\nu} y(x,t)+ \eta(x,t)=0,&\,x\in\Gamma_1, t>0,\\
\eta _t(x,t)-y_t(x,t)=0,&\, x\in\Gamma_1, t>0,\\
y(x,0)=y_0(x),\, y_t(x,0)=y_1(x), &  \,x\in\Omega,\\
\eta(x,0)=\eta_0(x)&\,x\in\Gamma_1.\\
\end{array}
\right.\\
$$
 The initial value problem associated to the conservative system $(P_0)$ is given by 
\be \label{P0}
\hspace{4cm}
 u_t=\caA_c u, \,\,\, u(0)=u_0,
\ee
where 
$$ \caA_c u =(z,\xi, \Delta y, -\p \nu y-\eta,z|_{\Gamma_1}),\, \forall u=(y,\zeta,z,\xi,\eta) \in {\mathcal D}({\mathcal A_c}),\,\,\caD(\caA_c)=\caD(\caA_d) .
$$
As the operators $D$ and $D^*$ given by,
$$D\eta =D^*\eta=\sqrt{b} \eta,\, \eta\in L^2(\Gamma_1),$$
are bounded, Lemma \ref{estu2} holds true. Thus in order to deduce the polynomial stability of the solution of (\ref{Pb}), it is sufficient to check that the solution $u_1$ of (\ref{P0}) satisfies the observability inequality $(O)$,
 $$b\inttbn  |\eta_1 ^2|\gtrsim \|u_0\|_{\caD(\caA_c^{-2})}^2,$$ 
 where $\caD(\caA_c^{-2})$ denotes throughout the example the space $(\caD(\caA_c^{2}))^\prime$.\\
 We first state the following proposition.
\blem\label{hb}
Let $u_0=(y_0,\zeta_0,z_0,\xi_0,\eta_0)^{\top}\in \caH$ and let $u_1=(y_1,\zeta_1,z_1,\xi_1,\eta_1)^{\top}$ be the corresponding solution of the problem $(\ref{P0})$. Then there exists $c_T>0$ such that 
\be\label{obser}\inttbn  |\eta_1 ^2|\ge c_T \|u_0\|_{\caD(\caA_c^{-2})}^2.\ee 
\elem
 \begin{proof}\textbf{First step.} Let $v_0\in \caD(\caA_c)$ and $v=(y,\zeta,z,\xi,\eta)^{\top}$ be a solution of \be\label{Pv}v_t=\caA_c v,\,\, v(0)=v_0.\ee  Then there exists $T>0$ such that
\be\label{2ineq1} 
\|v_0\|_{\caH}^2 \leq   C_1 \inttbn |y_t|^2 +
C_2 \inttbn |\p{\nu} y|^2+C_3  \inttbn |\eta |^2,
\ee
for some positive constants $C_1, C_2, C_3$.\\


Indeed, for $v_0\in \caD(\caA_c)$, we have
\begin{eqnarray}\label{o1}
\intta y_{tt} (2  m\cdot\g y)&=&-\intta y_{t}  (2  m\cdot\g y_t)+\intabt{y_t 2 m \cdot\g y}\\
&=&2\intta |y_t|^2 -\inttb (m\cdot\nu) |y_t|^2+\intabt{y_t 2 m\cdot \g y},\nonumber\end{eqnarray}
and
\begin{eqnarray}\label{o2}
\intta \Delta y (2 m \cdot\g y) &=&- \intta \g y\cdot  \g (2 m\cdot \g y) + \inttb \p{\nu} y (2 m\cdot \g y)\\
&=&-\inttb (m\cdot\nu) |\g y|^2+ \inttb \p{\nu} y (2 m\cdot \g y).\nonumber\end{eqnarray}

Finally, multiplying the wave equation by $2  m\cdot\g y$ and  substracting (\ref{o2}) from (\ref{o1}) leads to,
\begin{eqnarray}\label{r1} 2\intta |y_t|^2-\inttbn (m\cdot\nu) |y_t|^2+ \inttb (m\cdot\nu) |\g y|^2&&\\
+\intabt{ y_t 2 m \cdot\g y}-\inttb \p{\nu} y (2 m\cdot \g y)=0.\nonumber
\end{eqnarray}
Multiplying the wave equation equation by $y$ we obtain 
\be \label{r2} -\intta |y_t|^2+  \intta |\g y|^2+ \intabt{y_t y}-\inttb (\nu\cdot\g y) y=0.\ee
As
$$\inttb \p{\nu} y (2 m\cdot \g y)=2 \inttb (m.\nu)( \p{\nu} y)^2+2 \inttb (m.\tau)(\p{\nu} y \p{\tau} y),
$$ 
then taking into consideration the Dirichlet condition on $\Gamma_0$, we get 
\begin{eqnarray*}
\inttb \p{\nu} y (2 m\cdot \g y)- \inttb (m\cdot\nu) |\g y|^2
&=& \inttb (m.\nu)( \p{\nu} y)^2-
 \inttbn (m.\nu)( \p{\tau} y)^2\\
&&+2 \inttbn (m.\tau)( \p{\nu} y \p{\tau} y).
\end{eqnarray*}
Due to the geometric conditions imposed on $\Gamma$, we have
\begin{eqnarray} \label{r3} \inttb \p{\nu} y (2 m\cdot \g y)- \inttb (m\cdot\nu) |\g y|^2&\leq& \inttbn (m.\nu)( \p{\nu} y)^2+\frac{R^2}{\omega} \inttbn ( \p{\nu} y)^2\\
&\leq& (R+\frac{R^2}{\omega}) \inttbn ( \p{\nu} y)^2.\nonumber\end{eqnarray}

Hence (\ref{r1}) leads to 

\be\label{r6} 2\intta |y_t|^2+\intabt{ y_t 2 m \cdot\g y}\leq \inttbn (m\cdot\nu) |y_t|^2+(R+\frac{R^2}{\omega}) \inttbn ( \p{\nu} y)^2. \ee
Adding  (\ref{r2}) to (\ref{r6}), we obtain

\begin{eqnarray}\label{r7}
\intta |y_t|^2+ \intta |\g y|^2+\intabt{ y_t 2 m \cdot\g y} + \intabt{y_t y}-\inttb (\nu\cdot\g y) y\\
\leq \inttbn (m\cdot\nu) |y_t|^2+(R+\frac{R^2}{\omega}) \inttbn ( \p{\nu} y)^2.\nonumber
\end{eqnarray}

Note moreover that $\intabt{ y_t 2 m \cdot\g y}+ \intabt{y_t y} \gtrsim - E(0)$, and 
$$\inttbn \p{\nu} y y\leq \dfrac{1}{2\eps } \inttbn (\p{\nu} y)^2+  \dfrac{\eps}{2} \inttbn  y^2\leq \dfrac{1}{2\eps } \inttbn (\p{\nu} y)^2+\dfrac{c_p \eps}{2} \intta |\g y|^2. $$
We deduce that  for $\epsilon>0$ chosen small enough there exists $C>0$ such that 
\begin{eqnarray}\label{r5}(T-C) \|v_0\|_{\caH}^2 -\frac1a \inttbn{|y_t|^2}-\inttbn{|\eta|^2}&\leq& \inttbn (m\cdot\nu) |y_t|^2+ \dfrac{1}{2\eps } \inttbn (\p{\nu} y)^2\\
&&+(R+\frac{R^2}{\omega}) \inttbn ( \p{\nu} y)^2.\nonumber\end{eqnarray}
Finally, choosing $T$  large enough, we get the required result (\ref{2ineq1}).\\

\textbf{Second step.} Let $\alpha>0$ and  set 
 $$\eta_1=\dfrac{1}{a}(-\p \nu y-\eta)+ 2 \alpha z + \alpha ^2 \eta$$

We have the following expression for $|\eta_1|^2$ on $\Gamma_1$,
$$|\eta_1|^2 =\dfrac{1}{a^2}|\p{\nu} y|^2+ 4 \alpha^2|z|^2+\dfrac{(a \alpha^2-1)^2}{a^2}|\eta| ^2 -\dfrac{4 \alpha }{a} \p{\nu} y z + \dfrac{2}{a^2} (1-a \alpha^2) \p{\nu} y \eta+ \dfrac{4\alpha}{ a}(a \alpha^2-1) z \eta.$$ 
 
By the boundary condition on $\Gamma_1$, $\eta_t= z$ and $\p \nu y=-\eta -a \eta_{tt}$,
we get 

$$|\eta_1|^2 =\dfrac{1}{a^2} |\p \nu y|^2+  4 \alpha^2|z|^2+\dfrac{(-1+a \alpha ^2) (1+a \alpha ^2)}{a^2}|\eta| ^2 +4 \alpha^3 \eta \eta_t  +  4 \alpha \eta_t \eta_{tt} + \dfrac{2 \left(-1+a \alpha^2\right) }{a}\eta \eta_{tt}.$$

Thus 

\begin{eqnarray}\label{r17}
\dint_0^T  \intbn |\eta_1|^2ds&=& \dint_0^T  \intbn [ \dfrac{1}{a^2} |\p \nu y|^2+( 4 \alpha^2-\dfrac{2(-1+a \al^2)}{a})|z(1)|^2+\dfrac{(-1+a \alpha ^2)(1+a \alpha ^2)}{a^2}|\eta| ^2]ds \nonumber\\
&&+2\al ^3\intbn(\eta(T)^2-\eta(0)^2)ds+2 \al \intbn( \eta_t(T)^2-\eta_t(0)^2)ds + \\
&&\dfrac{2(-1+a \al^2)}{a} \intbn \eta(T)\eta_t(T)ds -\dfrac{2(-1+a \al^2)}{a} \intbn\eta(0)\eta_t(0))ds.\nonumber 
\end{eqnarray}

Choosing $\al$ large enough, we get $$w_1=\dfrac{1}{a^2}>0,\;  w_{2}=4\al^2-  \dfrac{2(-1+a \al^2)}{a}=\frac{2 (1+a \al ^2)}{a}>0, w_{3}=\dfrac{(-1+a \al^2) (1+a \al^2)}{a^2}>0.$$ 
In addition, (\ref{r17}) implies that
$$\begin{array}{ll}
\dint_0^T \intbn |\eta_1|^2 ds\geq &\dint_0^T \intbn [ w_1 |\p \nu y|^2+ w_{2}|z(1)|^2+w_{3}|\eta| ^2] ds -K_{a,\alpha} \|v_0\|_{\caH}^2, 
\end{array}
$$
for some constant $K_{a,\alpha}\geq 0 $  independent of $T$.

Combining the previous inequality with (\ref{2ineq1}), we deduce the existence of $c_1>0$ 
such that 

$$\dint_0^T \intbn |\eta_1|^2ds\geq  
c_1 (T-2) \|v_0\|_{\caH}^2 -K_{a,\alpha} \|v_0\|_{\caH}^2.$$
Finally, choosing $T$ large enough, we obtain
\be\label{observineq} \dint_0^T \intbn |\eta_1|^2 ds=
\dint_0^T \intbn \left| \dfrac{1}{a}(-\p \nu y-\eta)+ 2 \alpha z(1) + \alpha ^2 \eta \right|^2 ds \geq  
c_2 \|v_0\|_{\caH}^2 ,\ee
for some positve constant $c_2$ depending on $T$.

\textbf{Last  step.}
Let $u_0\in D(\caA_d)$ and let $u_1=(y_1,\zeta_1,z_1,\xi_1,\eta_1)^{\top}$ be the corresponding solution of (\ref{P0}),  then $$v=(y,\zeta,z,\xi,\eta)^{\top}= [(\caA_c +\alpha I) ^2]^{-1} u,$$ is a solution of (\ref{Pv}) where $v_0= [(\caA_c +\alpha I) ^2]^{-1} u_0 \in \caD(\caA_c)$. Since $(\caA_c +\alpha I) ^2=\caA_c^2 + 2 \alpha  \caA_c +\al ^2 I,$  
the last component $\eta_1$ of $u_1$ 
is given by $$\eta_1=\dfrac{1}{a}(-\p \nu y-\eta)+ 2 \alpha z(1) + \alpha ^2 \eta,$$ thus by the two previous steps we get (\ref{observineq}). Noting that  $\|u_0\|_{\caD(\caA_c)}\sim\|v_0\|_{\caH}$,
we consequently deduce that (\ref{obser}) holds for all $u_0\in \caD(\caA_c)$.
\end{proof}
\begin{theorem}
Let $u_0\in \caD(\caA_d)$ and let $u$ be the solution  of (\ref{Pb}). Then $u$ satisfies,
\be\label{dpexb2}\|u(t)\|^2\lesssim \dfrac{1}{(1+t)^\frac{1}{2}}
\| u_0\|_{\caD(\caA_d)}^2.\ee
\end{theorem}
\begin{proof}
Since the operator $D\in\caL(U)$, Lemma \ref{estu2} holds true.\\

Set  ${\caH}_1=D(\caA_c)$ and ${\caH}_2=D(\caA_c^{-2})$. Then $\caH=[{\caH}_1; {\caH}_2]_{1/3}$.
By Lemma \ref{hb}, we have $$ \dint_0^T \|D^*u_1(s)\|_U^2ds \geq c_T \|u_0\|_{{\caH}_2}^2.$$ By Theorem \ref{Thpolynomial} applied for $\theta=1/3$, we therefore obtain  (\ref{dpexb2}).
\end{proof}
\brk
Using the same method we get an analogous result for the one dimensional problem.  we can also get the observability inequality by a spectrum analysis and that was already done 
in the paper \cite{Mercier:11}, where the authors obtained an optimal decay, thus we expect the decay in the two dimensional case to be optimal as well.
\erk
\brk
Consider the following system studied in \cite{abbas:13}
\begin{eqnarray}\label{zb13}
\left\{\begin{array}{lcl}
y_{tt}(x,t)-y_{xx}(x,t)&=&0,\quad 0<x<1,t>0,\\
y(0,t)&=&0,\quad t>0,\\
y_{x}(1,t)+(\eta(t),C_0)_{\mathbb C^n}&=&0,\quad t>0,\\
\eta_{t}(t)-B_0\eta(t)-C_0y_{t}(1,t)&=&0,\quad t>0,
\end{array}\right.
\end{eqnarray} 
and 
$$y(x,0)=y_0(x),y_{t}(x,0)=y_1(x), \eta(0)=\eta_{0}, 0<x<1,
$$
where $B_0\in M_n(\C)$, $C_0\in\C^n$ are given. System (\ref{zb13})
can be written in the form (\ref{Gs0}) where $V=\{y\in H^{1}(0,1):y(0)=0\},$ $X=L^2(0,1)$ and $U=\C^n$. In this case, $\widehat{C}=B_0$  is a bounded operator and $B\eta=(\eta,C_0)\delta_1$ for all $\eta\in\C^n$. Indeed, since $\widehat{C}$ is bounded then it is enough to verify assumption (O). Assumption (O) was verified in \cite{abbas:13} and the polynomial stability of (\ref{zb13}) was deduced. In particular, for $n=1$ we obtain the system studied in \cite{wehbe:03}, where a polynomial decay is proved using a mutltiplier method. The polynomial decay can be also obtained by proving an observability inequality for the solutions of the corresponding conservative system which is exactly what has been verified in \cite{abbas:13}, thus applying the appraoch intoduced in this paper. 
\erk
\section{Unbounded example}
Consider  the following system
\be \label{unbd} 
\left \{ \begin{array}{lll}
u_{tt}(x,t)-  u_{xx}(x,t)+ w(x,t)=0, &&t\in [0, \infty ),\, 0<x<1\\
w_t(x,t)-i w_{xx}(x,t)+w(\xi,t)\delta_{\xi}-u_t(x,t) =0,&& t\in [0, \infty ),\, 0<x<1 \\
u(0,t)=u(1,t)=w(0,t)=w(1,t)=0, \, t\in [0, \infty ),\\
u(x,0) = u_0 (x), \, \partial_t u (x,0) = u_1 (x), \, w(x,0) = w_0, \, 0 < x < 1,
\end{array} 
\right. 
\ee
where $\xi\in(0,1)$.
Define the following spaces and operators:
$$X=U=L^2(0,1), V=H^1_0(0,1), U=L^2(0,1),W=\C,$$
$$A:u\in D(A)\to -u_{xx}\in L^2(0,1),\, D(A)=H^2(0,1)\cap H^1_0(0,1),$$
and $B=B^*=I_{U}=I_{L^2(0,1)}$.  In addition,
$$D:\eta\in\C\to\eta \delta_{\xi} \in (D(C))',\,\,D^*:w\in D(C)\to w(\xi)\in\C,$$ and 
$$C:w\in D(C)\to -i w_{xx}\in L^2(0,1),\, D(C)=H^2(0,1)\cap H^1_0(0,1).$$
The operator $\widehat{C}$ is thus given by 
$$\widehat{C}w=i w_{xx}-w(\xi)\delta_{\xi}$$
with
$$D(\widehat{C})=\{w\in H^1_0(0,1)\cap [H^2(0,\xi)\cap H^2(\xi,1)]: i[w_x]_{\xi}=i[w_x(\xi^{+})-w_x(\xi^{-})]=w(\xi)\}.$$
As the operator $D$ is unbounded, we need to verify that the problem satisfies assumption $(H)$ as well as  the asumption (O) for conservative problem.
In this case we have 
$${\cal D}(\caA_d)=D(A)\times V \times D(\widehat{C})$$ and 
$${\cal D}(\caA_c)=D(A)\times V \times D(C).$$
In order to verify the assumption $(H)$, we proceed by  finding the transfer function, for this purpose we recall that $u_2=u-u_1$ and $w_2=w-w_1$ satisfies (\ref{Ds1}) which is in this case
\be 
\label{unbd2} 
\left \{ \begin{array}{lll}
\p{tt} u_2- \p{xx}  u_2(x,t)+ w_2(x,t)=0, &&t\in [0, \infty ),\, 0<x<1\\
\p t w_2(x,t)-i \p{xx}  w_2(x,t) - \partial_t u_2=w(\xi,t)\delta_{\xi},&& t\in [0, \infty ),\, 0<x<1 \\
u_2(0,t)=u_2(1,t)=w_2(0,t)=w_2(1,t)=0, \\
u_2 (x,0) = 0, \, \partial_t u_2 (x,0) = 0, \, w_2 (x,0) = 0, \, 0 < x < 1,
\end{array} 
\right. 
\ee
and
\be 
\label{unbd2bis} 
\left \{ \begin{array}{lll}
\p{tt} u_1 - \p{xx}  u_1(x,t)+ w_1(x,t)=0, &&t\in [0, \infty ),\, 0<x<1\\
\p t w_1(x,t)-i \p{xx}  w_1(x,t) - \partial_t u_1 = 0,&& t\in [0, \infty ),\, 0<x<1 \\
u_1(0,t)=u_1(1,t)=w_1(0,t)=w_1(1,t)=0, \, t\in [0, \infty ),\\
u_1 (x,0) = u_{1,0}, \, \partial_t u_1 (x,0) = u_{1,0}, \, w_2 (x,0) = w_{2,0}, \, 0 < x <1,
\end{array} 
\right. 
\ee
Verifying the  assumption $(H)$ is equivelant to verifying (see \cite[Proposition 3.2]{ammari:01} for more details)
$$|{w}_{2}(\xi,t)|^2 \lesssim |{w}(\xi,t)|^2.$$ 
For this purpose, we state the following proposition.
\bprop\label{Transunbd}
Let $(u_2,w_2)=(u-u_1,w-w_1)$ be the solution of (\ref{unbdl}). Then $w_2$ verifies $$|{w}_{2}(\xi,t)|^2 \le |{w}(\xi,t)|^2.$$
\eprop
\begin{proof}
Let $\lambda =1+i\eta$ and consider $\hat{u}_2, \hat{w}_2$ the Laplace transforms of  $u_2$ and $w_2$  respectively. Then  $\hat{u}_{2}$ and $\hat{w}_{2}$ satisfies (\ref{Ls}) given by 
\be \label{unbdl}  
\left \{ \begin{array}{lll}
\lambda^2\hat{u}_2(x,\lambda)- \p{xx}\hat{u}_{2}(x,\lambda)+ \hat{w}_2(x,\lambda)=0,\\
\lambda
\hat{w}_{2}(x,\lambda)-i \p{xx}\hat{w}_{2}(x,\lambda)-\lambda \hat{u}_{2}=\hat{w}(\xi,\lambda)\delta_{\xi},
\end{array} 
\right. 
\ee
The problem reduces to studying $\hat{u}_2$ and $\hat{w}_2$ solutions of 
\be\label{q1} \left\{\begin{array}{l}
\lb^2 \uu -\p{xx} \uu +\w =0 \\ 
\lb\w - i \p{xx} \w - \lb \uu =-i \delta _{\xi}
\end{array} \right.
\ee
with $$\hat{u_2}(0)=\hat{u_2}(1)=0, \hat{w_2}(0)=\hat{w_2}(1)=0, [\partial_x\hat{ w_2}]_{\xi}=1, [\hat{ w_2}]_{\xi}=0,$$
and proving the existence of $C_{\beta}>0$
$$|w_2(\xi ,\lambda|)\le C_{\beta}, \forall \lambda=\beta +iy, y\in\R.$$

First, we set $$\w=\ww+\www,$$ where 
\be \label{qw3}\lb\ww - i \p{xx} \ww  =-i \delta _{\xi},\ee
with\be \label{qw3I}\ww(0)=\ww(1)=0, [\partial_x\ww]_{\xi}=1, [\ww]_{\xi}=0.\ee
and 
\be \label{q4}
\lb\www - i \p{xx} \www =  \lb \uu.
\ee
with \be \label{q4I}\www(0)=\www(1)=0.\ee
Let $\beta>0$ be fixed. It is required then to prove that 
$$|\ww(\xi ,\lambda)|\le C_{1\beta} ,\, |\www(\xi ,\lambda)|\le C_{2\beta}, \forall \lambda=\beta +iy, y\in\R.$$
We start by writing the expression of $\ww$,
$$\ww (x,\lb )= -\dsum_{k=1}^{+\infty} \frac{ \pkxi}{k^2 \pi^2-i \lb}\pk= -2\dsum_{k=1}^{+\infty}\frac{ \pkxii}{k^2 \pi^2-i \lb}.$$
For simplicity we consider $\lb = 1 \pm i \pi^2 y^2.$ 
$$|\ww (\xi,\lb )|\lesssim \dsum_{k=1}^{+\infty}\frac{ 1}{|(k^2\pm y^2)\pi^2 - i|}. $$
We first give an estimate for   $\lb = 1 + i \pi^2 y^2,$ 
$$|\ww (\xi,\lb )|\lesssim \dsum_{k=1}^{+\infty}\frac{ 1}{(k^2+y^2)\pi^2} \leq\dfrac{1}{6}.$$
For   $\lb = 1 - i \pi^2 y^2,$  we have
$$|\ww (\xi,\lb )|\leq \dfrac{2}{\pi^2} \left(  \dsum_{1\leq k\leq E(y)-1}
\frac{ 1}{y^2-k^2}+2\pi^2+\dsum_{E(y)+2\leq k}
\frac{ 1}{k^2-y^2} \right) .$$
But 
$$\dsum_{1\leq k\leq E(y)-1}
\frac{ 1}{y^2-k^2}\leq \sum_{1\leq k\leq E(y)-1}
\frac{E(y)-k}{E(y)^2-k^2}=\sum_{1\leq k\leq E(y)-1}
\frac{1}{E(y)+k}\leq 1.$$
and 
$$\dsum_{E(y)+2\leq k}
\frac{ 1}{k^2-y^2}=\dsum_{ k=2}^\infty \frac{ 1}{(k+E(y))^2-y^2}\leq \dsum_{ k=2}^\infty \frac{ 1}{(k-1)^2}=\dfrac{\pi^2}{6}. $$

Therefore  $|\ww (\xi,\lb )|$ is bounded on the line $\Re(\lb)=1.$\\

It remains to find the  estimate satisfied by  $\www(\xi,\lb)$. Indeed, since $(\pk)_{k\in\N^*}$ form a Hilbert basis of $L^2(0,1),$ then we may write $\uu, \w, \www$ as follows

$$\uu (x,\lb )=\dsum_{k=1}^{+\infty}   u_2^{(k)} \pk, \w=\dsum_{k=1}^{+\infty}   w_2^{(k)} \pk , \www=\dsum_{k=1}^{+\infty}   w_4^{(k)} \pk.$$
By the first equation of (\ref{q1}), we get
$$\forall k\geq 1,  u_2^{(k)}=  \dfrac{w_2^{(k)}}{k^2 \pi^2+ \lb^2}.$$
Due to (\ref{q4})
$$\forall k\geq 1,     w_4^{(k)}=\dfrac{\lb u_2^{(k)}}{ ik^2 \pi^2+ \lb}.$$
We deduce that 
$$w_4^{(k)}=-\dfrac{\lb w_2^{(k)}}{(k^2 \pi^2+ \lb^2)(ik^2 \pi^2+ \lb)} .$$
For $\lb=1+iy$ we have 
$$|k^2 \pi^2+ \lb^2|=\sqrt{4 y^2+\left(1+k^2 \pi ^2-y^2\right)^2}\geq 2|y|,$$
and
$$|ik^2 \pi^2+ \lb|=|1+i k^2 \pi ^2+i y|\geq |y|.$$ 
Hence for $|y|$ large enough we have
$$ |w_4^{(k)} |\leq \dfrac{|w_2^{(k)}|}{|y|}.$$
Using $\w=\ww+\ww$ we get for $|y|$ large enough 
$$ |w_4^{(k)} |\leq \dfrac{|w_3^{(k)}|}{|y|}.$$ 
We finally conclude that for $|y|$ large enough  $|\www(\xi,\lb)|$ is bounded on  the line $\Re(\lb)=1.$
It follows that $|\w(\xi,\lb)|$ is bounded as well.\\
\end{proof}
In what follows we prove that the observability assumption (O) holds on subspaces of $\caD(\caA_d)$ on which we deduce the polynomial stability of the energy.
Let us first remark that $0$ is not an eigenvalue of $\caA_d$. Let  $\lb=i \mu$ an eigenvalue of $\caA_d$  and $U=(u,v,w)$ a corresponding eigenvector. We then  have,

\be\label{obs1} \left\{\begin{array}{l}
-\mu^2 u -\p{xx} u +w =0 \\ 
\mu w -  \p{xx} w - \mu u = i w(\xi) \delta_{\xi} . 
\end{array} \right.
\ee
with $$u(0)=u(1)=w(0)=w(1)=0.$$

Multiplying the second equation by  $\overline{w}$ then integrating by parts on $(0,1)$, we find that $w(\xi)=0.$ We hence deduce that $w=0$. Moreover, multiplying the first equation by $\overline{u}$, integrating by parts and considering the imaginary part we deduce that $u=0$.

In order to verify the observability assumption (O) we study in what follows the spectrum of $\caA_c.$ Recall that the eigenvalues of $\caA_c$ are of the form $\lb=i\mu, \mu\in \R$.
\bprop\label{unbdegn} Let $\sigma(\caA_c)$ be the set of eigenvalues of $\caA_c$. 

(i) Then every element of $\sigma(\caA_c)$ is simple and $\sigma(\caA_c)$ is a disjoint union of three sets: 
$$\sigma(\caA_c)=\sigma_0\cup \sigma_1\cup \sigma_2$$ where  
$\sigma_0$ is a finite set, and there exists $k_0 \in\N^*$ such that 
$\sigma_1=\lbrace i \mu_{k,1}\rbrace_{k\in \Z,|k|\geq k_0},$ and
 $\sigma_2=\lbrace i \mu_{k,2}\rbrace_{k\in \N, k \geq k_0}.$   

(ii) For $i\mu_{k,i} \in \sigma_i, i=1,2,$ an associated eigenvector 
$\phi_{\mu_{k,i}}=\dfrac{1}{|k|^{\alpha_i}}(u_{\mu_{k,i}}, v_{\mu_{k,i}}, w_{\mu_{k,i})},$ with 
$\alpha_1= 1$ and $\alpha_2= 4$ is given by
$$u_{\mu_{k,i}}(x)=\sin(k\pi x), v_{\mu_{k,i}}(x)=i \mu_{k,i}u_{\mu_{k,i}}(x),
w_{\mu_{k,i}}(x)=(\mu^2_{k,i} -k^2\pi^2)\sin(k\pi x).$$

(iii) The following estimates  hold  

\be\label{estm1} 
\mu_{k,1}=k \pi+\dfrac{1}{2\pi^2 k^2}+o(\dfrac{1}{k^2}) ,\; |k|\rightarrow \infty,  
 \ee
\be\label{estf1} 
\Vert \phi_{\mu_{k,1}}\Vert _{\caH}\sim 1,
\ee
\be\label{wkxi1} 
\mu_{k,1}^2-k^2\pi^2=\dfrac{1}{k\pi}+o(\dfrac{1}{k}),
\ee
\be\label{estm2} 
\mu_{k,2}=-k^2 \pi^2 +O(\dfrac{1}{k^2}) ,\; k\rightarrow +\infty ,
 \ee
 \be\label{estf2}  
\Vert \phi_{\mu_{k,2}}\Vert _{\caH}=O(1), 
 \ee
 \be\label{wkxi2} 
\mu_{k,2}^2-k^2\pi^2=k^4\pi^4+O(k^2).
\ee
\eprop 
\begin{proof}
Let  $\lb=i\mu$ be an eigenvalue of $\caA_c$ and $U=(u,\lambda u,w)$ be a corresponding eigenvector of $\caA_c$. Then $u$ and $w$ satisfies 
\be\label{obs2} \left\{\begin{array}{l}
-\mu^2 u -\p{xx} u +w =0 \\ 
\mu w -  \p{xx} w - \mu u = 0 . 
\end{array} \right.
\ee
Replacing $w$ in the second equation, we find that 
\be\label{obs3}\left\{\begin{array}{l}
\p{xxxx} u + (\mu ^2- \mu) \p{xx}u   +(\mu-\mu ^3)u =0,\\ 
u(0)=\p{xx}u(0)= u(1)=\p{xx}u(1)=0.
\end{array} \right.
\ee
It is easy to check that $\mu=0,\mu=1$ and $\mu=-1$ are not eigenvalues of $\caA_c$.

Let $X_1=\dfrac{1}{2} (\mu -\mu ^2-\sqrt{\Delta })$ and 
$X_2=\dfrac{1}{2} (\mu -\mu ^2+\sqrt{\Delta })$ be the roots of 
$$p(X)=X^2  + (\mu ^2- \mu) X  +\mu-\mu ^3 =0,$$ 
where $\Delta=\mu  (\mu-1 )  (\mu^2+3 \mu +4)$ is the discriminant of $p$.

Set $t_i=\sqrt{X_i},i=1,2$ then the general form of $u$ satisfying the first equation
of (\ref{obs3}) and the left boundary condition is  
$$u(x)=c_1 \sinh(t_1x)+c_2 \sinh(t_2x ).$$
Considering the right boundary conditions we see that $u$ is non trivial if and only if $t_1$ and $t_2$ satisfy 
$$\sinh(t_1) \sinh(t_2)(t_1^2-t_2^2)=0.$$
But $t_1^2-t_2^2\neq 0$, since $\mu\ne 0$ and $\mu \ne 1$. Hence $t_1$ and $t_2$ satisfy the following characteristic equation  

$$\sinh(t_1) \sinh(t_2)=0,$$ 
which gives that $t_1=ik \pi$ or  $t_2=i k \pi, k\in\Z^*$ i.e 
$X_1=-k^2\pi^2$ or $X_2=-k^2\pi^2.$ 

Now, we remark that all the eigenvalues of $\caA_c$ are  simple.
Suppose otherwise that there exists a  double eigenvalue, then there exist $k_i, \in \N^*, i=1,2 $ s.t $X_i=-k_i \pi^2,i=1,2.$ Thus we have 
$$\dfrac{X_1X_2}{X_1+X_2}=-\dfrac{k_1^2 k_2^2 \pi ^2}{k_1^2+k_2^2}=\mu+1.$$
Now, replacing $\mu$ in $X_1+X_2=\mu -\mu ^2$, we find that 
$$2 k_1^4+4 k_1^2 k_2^2+2 k_2^4-k_1^6 \pi ^2-k_2^6 \pi ^2+k_1^4 k_2^4 \pi ^4=0,$$  
which is impossible since $\pi^2$ is a transcendental number.

Therefore,
$$u(x)=\sin(k\pi x), \; w(x)=(\mu^2+X_i)\sin(k\pi x) ,\; i=1 \mbox{ or }  2.$$
Moreover, 
the eigenvalues of $\caA_c$ are formed of two disjoint families of eigenvalues. The first class of eigenvalues is obtained from $X_1=-k^2\pi^2,$ the second class is obtained from $X_2=-k^2\pi^2.$

Now, we firstly study the asymptotic behaviour of the first class: since $X_1=-\mu^2+\dfrac{1}{\mu}+o(\dfrac{1}{\mu})=-k^2\pi^2$ then $\mu=k\pi +\dfrac{1}{2\pi^2 k^2}+o(\dfrac{1}{k^2}), |k|\rightarrow \infty.$ 
If we denote by $\lbrace i\mu_{k,1}\rbrace_{k\in\\Z^*}$ this first class of eigenvalues then the previous estimate is (\ref{estm1}). 
Using the previous estimate we directly get (\ref{estf1}) and (\ref{wkxi1}). 

Secondly, since $X_2=\mu+O(\dfrac{1}{\mu})=-k^2\pi^2$ we deduce that the large eigenvalues of the second class are negative, and denoting them by $i\mu_{k,2}$ we easily see that (\ref{estm2}) holds true. Moreover, since   $\mu_{k,2}^2-k^2\pi^2=O(k^4)$ then  (\ref{estf2}) holds.
\end{proof}
In order to use generalized Inghams inequalities we need to estimate 
$\displaystyle\inf_{\mu_{k,1}\in \sigma_1,\mu_{k',2}\in \sigma_2} |\mu_{k,1}-\mu_{k',2}|$.  Unfortunately it seems to be a difficult task and it remains a open question. Hence, to get an observability result we will take the initial condition $U_0$ in some subspaces of $\caH.$ So we introduce 
 
\begin{center}$H_1=\hbox{span}({\displaystyle\phi_{\mu})_{\mu\in \sigma_0}}\cup \hbox{span}{\displaystyle
(\phi_{\mu})_{\mu\in \sigma_1}}$ and $H_2=\hbox{span}({\displaystyle\phi_{\mu})_{\mu\in \sigma_0}}\cup \hbox{span}{\displaystyle
(\phi_{\mu})_{\mu\in \sigma_2}}.$

\end{center}
Before given an observability result we introduce the set $\cal{S}$ of all numbers $\rho \in (0,\pi)$  such that $\dfrac{\rho}{\pi} \notin \Q$ 
and if $[0,a_1 ,...,a_n ,...]$  is the expansion of $\dfrac{\rho}{\pi}$
as a continued fraction, then $(a_n)$ is bounded. Recall that if $\pi \xi \in  \cal{S}$ then 
\be \label{estsin}
|\sin(k\pi \xi)| \gtrsim \dfrac{1}{|k|},\; k\in \Z^*,
\ee (see for instance  \cite{ammari:00}). 

 \bprop\label{unbdobs}
 \begin{enumerate}
 \item
For all $\xi \in (0,1)$ there are not $T, C > 0$ such that for all $U_ 0 \in {\cal H}$ we have
 \be 
 \label{obd1bis}  \int_{0}^{T}|w(\xi,t)|^2 dt\geq C_T 
 \Vert U_0\|^2_{{\cal H}}.\ee 
 \item
Suppose that $\xi\in \cal{S}.$\\  Let $U_0\in H_1$ and $U=(u,v,w)$ be the corresponding solution of the conservative problem 
 \be\label{consunbd}
 U_t=\caA_c U, U(0)=U_0.
 \ee 
 Then there exists T>0 and a constant $c_T>0$ such that 
 \be \label{obd1}  \int_{0}^{T}|w(\xi,t)|^2 dt\geq C_T 
 \Vert U_0\|^2_{\caD(\caA_c^{-3})},\ee 
where $\caD(\caA_c^{-3}) = (\caD(\caA_c^{3}))^\prime,$ obtained by means of the inner product in $X$.

For $U_0\in H_2$ we have  

 \be \label{obd2}  \int_{0}^{T}|w(\xi,t)|^2 dt\geq C_T 
 \Vert U_0\|^2_{\caD(\caA_c^{-\frac{1}{2}})}.\ee 
 
\end{enumerate}
\eprop
\begin{proof}
\begin{enumerate}
\item
Since 
$$
 \lim_{n \rightarrow + \infty} \left\|\left(i\mu_{n,1} - \caA_c \right)\phi_{n,1}\right\|_\caH^2 +  \left\|\left(\begin{array}{lllc} 0 & 0 & D^* \end{array} \right) \phi_{n,1} \right\|^2_U = 0.
$$
Which implies according to \cite[Theorem 5.1]{miller:05} that we don't have the exact observability, i.e., the inequality (\ref{obd1bis}). 
\item
Let $U_0\in H_1$. We may write 
$$U_0=\sum_{\mu\in \sigma_0}u_0^{\mu}\phi_{\mu}+\sum_{|k|\ge k_0}u_{0}^{(k)}\phi_{\mu_{k,1}}.$$ Moreover,
$$w(\xi,t)=\dfrac{1}{|k|}\left(\sum_{\mu\in \sigma_0}u_0^{\mu}e^{i\mu t}w_{\mu}(\xi)+\sum_{|k|\ge k_0}u_{0}^{(k)}e^{i  \mu_{k,1} t }w_{\mu_{k,1}}(\xi) \right).$$
Note that $\gamma_1=\displaystyle\inf_{\mu,\mu'\in \sigma,\mu\neq \mu'}|\mu-\mu'|>0,$ then using Ingham's inequality there exists $T>{2\pi}{\gamma_1}>0$ and  a constant $c_T>$ depending on $T$ such that

$$\int_{0}^{T}|w_{1}(\xi,t)|^2dt\geq c_T\dfrac{1}{|k|}\left(\sum_{\mu\in \sigma_0}|u_{0}^{\mu}w_{\mu}(\xi)|^2+\sum_{|k|\ge k_0 }|u_{0}^{(k)}w_{\mu_{k,1}}(\xi)|^2\right).$$ 
Now using (ii), and estimates (\ref{estm1}),(\ref{estf1}), (\ref{wkxi1}) of Proposition \ref{unbdegn} we get (\ref{obd1}). 
For $U_0 \in H_2,$ we use analogous argument. 
\end{enumerate}
\end{proof} 
\begin{theorem}
\begin{enumerate}
\item
For any $\xi \in (0,1)$, the system described by (\ref{unbd}) is not exponentially
stable in ${\cal H}$.
\item
Let $U_0\in H_1\cap \caD(\caA_d),$ and let $U$ be the solution  of the corresponding dissipative problem
$$U_t=\caA_d U, \,\, U(0)=U_0.$$
 Then $U$ satisfies,
\be \label{dstab}
\|U(t)\|^2 \lesssim \dfrac{1}{(1+t)^{\frac{1}{3}}}\| U_0\|_{\caD(\caA_d)}^2.
\ee
\item
Let $U_0\in H_2\cap \caD(\caA_d),$ and let $U$ be the solution  of the corresponding dissipative problem
$$U_t=\caA_d U, \,\, U(0)=U_0.$$
 Then $U$ satisfies,
\be \label{dstab2}
\|U(t)\|^2 \lesssim \dfrac{1}{(1+t)^{2}}\| U_0\|_{\caD(\caA_d)}^2.
\ee
\end{enumerate}
\end{theorem}
\begin{proof}
\begin{enumerate}
\item
This result is a direct consequence of the first assertion of Proposition \ref{unbdobs} and Theorem \ref{THexp}.
\item
Due to Proposition \ref{Transunbd} and Proposition \ref{unbdobs} we deduce (\ref{dstab}) from Theroem \ref{Thpolynomial} setting    
${\caH}_1=\caD(\caA_c)$ and ${\caH}_2=\caD(\caA_c^{-3})$ and $\theta=\frac 14$. 
\item
As in 2.  we deduce (\ref{dstab2})  setting    
${\caH}_1=\caD(\caA_c)$ and ${\caH}_2=\caD(\caA_c^{-\frac{1}{2}})$ and $\theta=\frac 23$. 
\end{enumerate}
\end{proof}

\bibliographystyle{abbrv}
\bibliography{zainabdkz}

\edc